\documentclass[a4paper,11pt,reqno]{amsart}

\usepackage{amssymb,amsmath,amsfonts} 
\usepackage{a4wide}

\usepackage{graphicx}
\usepackage{pictex}
\usepackage{epstopdf}
\usepackage{etex}

\usepackage{hyperref}
\usepackage{color}
\usepackage{amsthm}

\usepackage{latexsym}
\usepackage{autograph,epic,latexsym,bezier,amsbsy,color,enumerate,amsfonts,amsmath,amscd,amssymb}
\usepackage[latin1]{inputenc}

\numberwithin{equation}{section}    
\theoremstyle{plain}
\newtheorem{Theorem}{Theorem}[section]
\newtheorem{Proposition}[Theorem]{Proposition}
\newtheorem{Corollary}[Theorem]{Corollary}
\newtheorem{Lemma}[Theorem]{Lemma}
\theoremstyle{definition}

\newtheorem{Definition}[Theorem]{Definition}

\theoremstyle{remark}
\newtheorem{Remark}[Theorem]{Remark}

\newcommand{\RR}{\mathbb{R}}
\newcommand{\CC}{\mathbb{C}}
\newcommand{\NN}{\mathbb{N}}
\newcommand{\TT}{\mathbb{J}}
\newcommand{\ZZ}{\mathbb{Z}}

\newcommand{\graph}{\mbox{gr}}
\newcommand{\Graph}{\mbox{Gr}}
\newcommand{\rank}{\mbox{rk}}

\newcommand{\complexity}{\mathcal{C}}

\newcommand{\Hmm}[1]{\leavevmode{\marginpar{\tiny%
$\hbox to 0mm{\hspace*{-0.5mm}$\leftarrow$\hss}%
\vcenter{\vrule depth 0.1mm height 0.1mm width \the\marginparwidth}%
\hbox to 0mm{\hss$\rightarrow$\hspace*{-0.5mm}}$\\\relax\raggedright
#1}}}

\begin{document}

\title[Schreier graphs of Grigorchuk's group and the  subshift of $\mathcal{T}$ ]{Schreier graphs of
Grigorchuk's group and  a subshift associated to a non-primitive
substitution}

\author{Rostislav Grigorchuk}
\address{Mathematics Department, Texas A\&M University, College Station, TX 77843-3368, USA} \email{grigorch@math.tamu.edu}

\author{Daniel Lenz}
\address{Mathematisches Institut \\Friedrich Schiller
Universit{\"a}t Jena \\07743 Jena, Germany }
\email{daniel.lenz@uni-jena.de}

\author{Tatiana Nagnibeda}
\address{Section de Math\'{e}matiques, Universit\'{e} de Gen\`{e}ve, 2-4, Rue du
Li\`{e}vre, Case Postale 64 1211 Gen\`{e}ve 4, Suisse}
\email{Tatiana.Smirnova-Nagnibeda@unige.ch}

\date{\today}

\begin{abstract}
There is a recently discovered connection between  the spectral
theory of Schr\"o-dinger operators whose potentials exhibit
aperiodic order and  that of Laplacians associated with actions of
groups on regular rooted trees, as  Grigorchuk's group of
intermediate growth. We give an overview of corresponding results,
such as different spectral types in the isotropic and anisotropic
cases, including Cantor spectrum of Lebesgue measure zero and
absence of eigenvalues. Moreover, we discuss the relevant background
as well as the combinatorial and dynamical tools that allow one  to
establish the afore-mentioned connection. The main such tool is the
subshift associated to a substitution over a finite alphabet that
defines the group algebraically  via a recursive presentation by
generators and relators.

\end{abstract}

\maketitle


\tableofcontents

\unitlength=0,4mm \textwidth = 16.00cm \textheight = 22.00cm
\oddsidemargin= 0.12in \evensidemargin = 0.12in
\setlength{\parindent}{8pt} \setlength{\parskip}{5pt plus 2pt minus
1pt} \setloopdiam{10}\setprofcurve{7}

\section*{Introduction}
The study of spectra of graphs associated with finitely generated
groups, such as Cayley graphs or Schreier graphs (natural analogues
of Cayley graphs associated to not necessarily free group actions) has a long history and is related to many
problems in modern mathematics. Still, very little is known about
the dependence of the spectrum of the Laplacian on the choice of generators in the group and on the weight on these generators. In the
recent paper \cite{GLN}  we addressed  the  issue of dependence on the weights on generators in the example of Grigorchuk's group of intermediate growth.
More specifically, we determined the spectral type of the Laplacian
on the Schreier graphs describing the action of Grigorchuk's group
on the boundary of the infinite binary tree and showed that it is
different in the isotropic and anisotropic cases. In fact, the
spectrum is shown to be a Cantor set of Lebesgue measure zero in the
anisotropic case, whereas, as has been known for a long time, it consists of one or  two intervals in
the isotropic case. Moreover, we showed almost surely (with respect to
a natural measure on the boundary of the tree) the absence of eigenvalues for the Laplacians in
question.

Our investigation in \cite{GLN} provides (and relies on)  a
surprising link between discrete Schr\"odinger operators with
aperiodic order and the substitutional dynamics arising from a
presentation of the group by generators and relators.

The purpose of the present paper is two-fold. Firstly, we want to
survey  the  spectral theoretic results of \cite{GLN}. Secondly, we
want to put these results  in wider perspective by discussing the
background on aperiodic order in dimension one. In this context we
present a discussion of subshifts and aperiodic order in Section
\ref{Section-Subshifts} and of  Schr\"odinger operators arising in
models of (dis)ordered solids  in Section \ref{Schroedinger}.

 We also continue our study of the substitution that is
instrumental for the results in \cite{GLN}. It first appeared in
1985 in the presentation of Grigorchuk's group by generators and
relators found by Lysenok \cite{Lys} (such infinite recursive
presentations are now called L-presentations). A remarkable fact
 is that the same substitution also
describes basic dynamical properties of the group. Here we review
its  combinatorial properties and  carry out  a  detailed study of
the factor map to its maximal equicontinuous factor, which is the
binary odometer. This factor map is proven  to be one-to-one
everywhere except on three orbits.  This in turn can be  linked to
the phenomenon of pure point diffraction which is at the core of
aperiodic order.
All these discussions concerning the substitution are contained in
Section \ref{The-substitution}.

The necessary background from graphs and dynamical systems is
discussed in Section \ref{Section-Background-Graphs-etc} and basics
on Grigorchuk's group and its Schreier graphs can be found in
Section \ref{Grigorchuk}.

The connection between Schr\"odinger operators and Laplacians on
Schreier graphs  revealed in \cite{GLN} (and reviewed in Sections
\ref{Connecting} and  \ref{Section-Spectral}) can also be used to
show that the Kesten-von-Neuman trace and the integrated density of
states agree. This seems to be somewhat folklore. We provide a proof
in Section \ref{IDS-Schreier}.

Our approach can certainly be carried out for various further
families of groups generalizing  Grigorchuk's group. In particular,
our results fully extend to a larger family of self-similar groups acting on the infinite binary tree considered by Sunic  in \cite{Sun}, as well as to similar families of self-similar groups acting on regular rooted trees of arbitrary degrees that also have linear Schreier graphs for their action on the boundary of the tree. Moreover, each of these self-similar groups belongs to an uncountable family of groups parametrized by sequences in a certain finite alphabet, in the same way as Grigorchuk's group belongs to the family of groups $(G_\omega)_{\omega\in\{0,1,2\}^{\bf N}}$ constructed by the first named author in \cite{Gri84}. The Schreier graphs of the groups in the same family look the same, but their labelling by generators, and thus their spectra, depend on the sequence $\omega$. The associated subshifts are also different, in particular, they don't have to come from a substitution in the case when the corresponding group is not self-similar. This more general setup will be considered in a separate paper.

\textbf{Acknowledgments.} R. G. was partially supported by the NSF
grant DMS-1207669,  by ERC AG COMPASP and by the  NSA grant
H98230-15-1-0328. The authors acknowledge support of the Swiss
National Science Foundation. Part of this research was carried out
while D.L. and R. G. were visiting the Department of mathematics of
the University of Geneva. The hospitality of the department is
gratefully acknowledged. The authors also thank Fabien Durand and
Ian Putnam for most enlightening discussions concerning  the
material gathered in Section \ref{section-recoding},  Yaroslav
Vorobets for allowing them to use his figures $3$ and $4$ and Olga
Klimecki for help in preparing figure $1$. Finally, the authors
would like to thank the anonymous referee for a very  careful
reading of the manuscript and several helpful suggestions.

\section{Subshifts and aperiodic order in one dimension
}\label{Section-Subshifts}
 Long range aperiodic order (or aperiodic
order for short)  denotes an intermediate regime of order between
periodicity and randomness. It has received a lot of attention over
the last thirty years or so, see e.g. the article collections and
monographs \cite{BGr,BM,KLS,Moody,Sen}.

This interest in aperiodic order  has various reasons. On the one
hand it is due  to the many remarkable and previously unknown
features and phenomena arising from aperiodic order in various
branches of mathematics. On the other hand, it is also due to the
relevance of aperiodic order in physics and chemistry. Indeed,
aperiodic order provides a mathematical theory for the description
of a new type of matter discovered in 1982 by Shechtman via
diffraction experiments \cite{SBGC}. These experiments showed sharp
peaks in the diffraction pattern indicating long range order and at
the same time a five fold symmetry indicating absence of
periodicity. The discovery of solids combining both long range order
with aperiodicity came as a complete surprise to physicists and
chemists and  Shechtman was honored with a  Nobel prize in 2011. By
now, the solids in question are known as quasicrystals.

In one dimension aperiodic order  is commonly modeled by subshifts
of low complexity. In higher dimensions it is modeled by dynamical
systems consisting of point sets with suitable regularity features
(which are known as Delone dynamical systems). Here, we present the
necessary notation in order to deal with the one-dimensional
situation. \textbf{When  we speak about aperiodic order
subsequently,  this will always mean that we have a  subshift  with
suitable minimality features  at our disposal.}

Let a finite set $\mathcal{A}$ be given. We call $\mathcal{A}$ the
\textit{alphabet} and refer to its elements as \textit{letters}. We
will consider the set $\mathcal{A}^\ast$ of finite words (including
the empty word) over the alphabet $\mathcal{A}$, viewed as a free
monoid (with the multiplication given by concatenation of words and
the empty word representing the identity). Elements of
$\mathcal{A}^\ast$ will often be written as $w = w_1\ldots w_n$ with
$w_j\in \mathcal{A}$. The length of a word is the number of its
letters. It will be denoted by $|\cdot|$. The empty word has length
zero. Then,  $\mathcal{A}^\ZZ$ denotes the set of functions from
$\ZZ$ to $\mathcal{A}$. We think of the elements of
$\mathcal{A}^\ZZ$ as of bi-infinite words over the alphabet
$\mathcal{A}$. The set $\mathcal{A}^\NN$ denotes the set of
functions from $\NN$ to $\mathcal{A}$. We think of its elements of
one-sided infinite words over $\mathcal{A}$. They will often be
denoted by $\xi = \xi_1\xi_2\ldots...$.

 If $v,w$ are finite words and $\omega\in \mathcal{A}^\ZZ$ satisfies
$$ \omega_1 \ldots \omega_{|v|} = v\:\; \mbox{and}\:\:
\omega_{-|w|+1}\ldots \omega_0 = w$$ we write
$$\omega = ... w|v...$$
and say that $|$ \textit{denotes the position of the origin}.

We   equip $\mathcal{A}$ with the  discrete topology and
$\mathcal{A}^{\ZZ}$ with the  product topology. By the Tychonoff
theorem, $\mathcal{A}^\ZZ$ is then compact. In fact, it is
homeomorphic to the Cantor set.  A pair $(\varOmega,T)$ is called a
\textit{subshift} over $\mathcal{A}$ if $\varOmega$ is a closed
subset of $\mathcal{A}^{\ZZ}$ which is invariant under the
\textit{shift transformation}
$$T :
\mathcal{A}^{\ZZ} \longrightarrow \mathcal{A}^{\ZZ}, \;\: (T \omega)
(n):= \omega (n+1).$$

If there exists a natural number $N\neq 0 $ with $T^N \omega =
\omega$ for all $\omega\in \varOmega$ then $(\varOmega,T)$ is called
\textit{periodic} otherwise it is called \textit{non-periodic}.

Whenever $\omega$ is a word over $\mathcal{A}$ (finite or infinite,
indexed by $\NN$ or by $\ZZ$) we define
$$\mbox{Sub} (\omega) := \mbox{Finite subwords of $\omega$}.$$
By convention, the set of finite subwords includes the empty word.
Every subshift  $(\Omega,T)$ comes naturally with the set
$\mbox{Sub} (\Omega)$ of associated finite words given by

 $$ \mbox{Sub} (\Omega)
:= \bigcup_{\omega\in \Omega} \mbox{Sub}(\omega).$$ A word $v \in
\mbox{Sub} (\Omega) $ is said to \textit{occur with bounded gaps} if
there exists an $L_v >0$ such that every $w\in \mbox{Sub} (\Omega)$
with $|w|\geq L_v$ contains a copy of $v$. As is well known (and not
hard to see)  $(\Omega,T)$ is minimal  if and only if every $v\in
\mbox{Sub} (\Omega)$ occurs with bounded gaps. For proofs and
further discussion we refer to standard textbooks such as
\cite{LindM,Wal}. We will be concerned here with the following
strengthening of the bounded gaps condition. It concerns the case
that $L_v$ can be chosen as $C |v|$ with  fixed $C$ (independent of
$v$).

\begin{Definition}[Linearly repetitive]  A subshift $(\varOmega,T)$ is
called \textit{linearly repetitive} (LR), if there exists a constant
$C>0$ such that any word $v \in \mbox{Sub}(\varOmega)$  occurs  in
any word $w\in \mbox{Sub}(\varOmega)$ of length at least $C |v|$.
\end{Definition}

\medskip

\begin{Remark} This notion has been discussed under various names by various
people. In particular it was studied  by Durand, Host and Skau
\cite{DHS} in the setting of subshifts  (under the name 'linearly
recurrent'). For Delone dynamical systems it was brought forward at
about the same time by Lagarias and Pleasants \cite{LP} under the
name 'linearly repetitive'. It has also featured in the work of
Boris Solomyak \cite{Sol} (under the name
 'uniformly repetitive').  It was also already discussed in an
unpublished work of Boshernitzan in the 90s. That work also contains
a characterization  in terms of positivity of weights. A
corresponding result for Delone systems was recently given in
\cite{BBL}.
\end{Remark}

\medskip

Durand \cite{Dur} gives a characterization of such subshifts in
terms of primitive $S$-adic systems and shows the following (which
was already known to Boshernitzan).

\begin{Theorem}\label{Theorem-unique-ergodicity} Let $(\varOmega,T)$ be a linearly repetitive
subshift. Then, the subshift is uniquely ergodic.
\end{Theorem}

\begin{Remark} In fact, linear repetitivity implies a
strong form of subadditive ergodic theorem \cite{Len3}. Validity of
such a result  together with the fundamental  work of Kotani
\cite{Kot} is at the core of the approach to Cantor spectrum of
Lebesgue measure zero developed in \cite{Lenz}. Our considerations
 on Cantor spectrum rely on an extension of that approach worked out in \cite{BP}.
\end{Remark}

\bigskip

\section{Schr\"odinger operators with aperiodic order}\label{Schroedinger}
In this section, we present (parts of) the spectral theory of
discrete  Schr\"odinger operators associated to minimal dynamical
systems.

Schr\"odinger operators occupy a prominent position in the theory of
aperiodic order. Indeed, they arise in the quantum mechanical
description of conductance properties of quasicrystals and exhibit
quite interesting mathematical properties. In fact, already the
first two papers on them written by physicists suggest that the
corresponding spectral measures are purely singular continuous and
the spectrum is a Cantor set of Lebesgue measure zero
\cite{KKT,OPRSS}. By now these features as well as other
conductance-related properties known as anomalous transport  have
been thoroughly studied in a variety of models by various authors,
see the survey articles \cite{Sut,Dam,DEG}. The phenomenon that the
underlying spectrum is a Cantor set of Lebesgue measure zero is
usually referred to as \textit{Cantor spectrum of Lebesgue measure
zero}  and this is how we will refer to it subsequently.

The Laplacians on Schreier graphs discussed later will turn out to
be  unitarily equivalent to certain such  Schr\"odinger operators.

Subsequently,  we first discuss basic mathematical  features of
Schr\"odinger operators associated to dynamical systems in the next
Section \ref{basic-schroedinger}, then turn to absolutely continuous
spectrum and the spectrum as a set for subshifts  in Section
\ref{spectrum-set-ac} and finally give some background from physics
in  Section \ref{aperiodic-order-schroedinger}.

\subsection{Constancy of the spectrum and the integrated density of states
(IDS)}\label{basic-schroedinger} In this section we review some
basic theory of discrete Schr\"odinger operators (or rather Jacobi
matrices) associated to minimal topological dynamical systems. This
framework is slightly more general than the framework of minimal
subshifts. The results we discuss  include constancy of the spectrum
and uniform convergence of the so-called integrated density of
states. All these  results are well known.

\medskip

Whenever $T$ is a homeomorphism of the compact space $\varOmega$ we
will refer to $(\varOmega,T)$ as a \textit{topological dynamical
system}. Later we will meet an even more general definition of
dynamical system. To continuous functions
 $f,g
: \varOmega \longrightarrow \RR$ we then  associate a family of
\textit{discrete operators} $(H_\omega)_{\omega\in \varOmega}$.
Specifically, for each $\omega \in\varOmega$,  $H_\omega$ is a
bounded selfadjoint operator from $\ell^2 (\ZZ)$ to $\ell^2 (\ZZ)$
acting via
$$(H_\omega u ) (n) = f(T^n \omega) u (n-1) +
f(T^{n+1}\omega)  u (n+1) + g(T^n \omega) u (n)$$ for $u \in \ell^2
(\ZZ)$ and $n\in\ZZ$.

In the case $f\equiv 1$ the above operators are known  as
\textit{discrete Schr\"odinger operators}. For general $f$ the name
\textit{Jacobi matrices} is often used in the literature. Here, we
will deal with the case $f\neq 1$ but still mostly refer to  the
arising operators as Schr\"odinger operators.

As the operator $H_\omega$ is selfadjoint, the operator $H_\omega -
z$ is bijective with continuous inverse $(H_\omega - z)^{-1}$ for
any $z\in\CC\setminus \RR$. Moreover, for any $\varphi \in \ell^2
(\ZZ)$ there exists a unique positive  Borel measure
$\mu_\omega^\varphi$ on $\RR$ with
$$\int_\RR \frac{1}{t -z} d\mu_\omega^\varphi (t) = \langle \varphi,
(H_\omega-z)^{-1} \varphi\rangle$$ for any $z\in\CC\setminus \RR$.
This measure if finite and assigns the value $\|\varphi\|^2$ to the
set $\RR$.

\smallskip

For fixed $\omega\in\varOmega$, the measures $\mu_\omega^\varphi$,
$\varphi\in\ell^2 (\ZZ)$, are called the \textit{spectral measures}
of $H_\omega$. The spectrum of $H_\omega$ is defined as
$$\sigma (H_\omega) :=\{ z\in\CC :  (H_\omega - z I) \;\mbox{lacks a
bounded two-sided inverse.}\}$$ It is the smallest set containing
the support of any $\mu_\omega^\varphi$  (see e.g. \cite{Wei}). The
spectrum is said to be \textit{purely absolutely continuous} if the
spectral measures for all $\varphi\in \ell^2 (\ZZ)$  are absolutely
continuous with respect to  Lebesgue measure. The spectrum is said
to be \textit{purely singular continuous} if all spectral measures
are both continuous (i.e. do not have discrete parts) and singular
with respect to Lebesgue measure.

The following result is well-known. It can be found in various
places, see e.g.   \cite{Lenz2}. Recall that $(\varOmega,T)$ is
called \textit{minimal} if  $\{T^n \omega : n\in\ZZ\}$ is dense in
$\varOmega$ for any $\omega \in\varOmega$.

\begin{Theorem}[Constancy of the spectrum] \label{Constancy}Let  $(\varOmega,T)$ be  minimal and $f,g
:\varOmega\longrightarrow \RR$ continuous. Then, there exists a
closed  subset $\Sigma \subset \RR$ such that the spectrum $\sigma
(H_\omega)$ of $H_\omega$ equals $\Sigma$ for any
$\omega\in\varOmega$.
\end{Theorem}

We will refer to the set $\Sigma$ in the previous theorem as the
\textit{spectrum of the Schr\"odinger operator associated to
$(\varOmega,T)$ (and (f,g))}. The spectrum $\Sigma$ is a one of the
main objects of interest in our study.

\medskip

Before turning to a finer  analysis of the spectrum we will
introduce a further quantity of interest the so called integrated
density of states. In order to do so, we will assume that the
underlying dynamical system $(\varOmega,T)$ is not only minimal but
also \textit{uniquely ergodic} i.e. possesses a   unique invariant
probability measure, which we call  $\lambda$.  Then, we can
associate to the family $(H_\omega)$ the positive  measure $k$ on
$\RR$ defined via
$$\int_{\RR}  F (x) dk (x) :=\int_\varOmega \langle f(H_\omega)
\delta_0, \delta_0\rangle d\lambda (\omega)$$ (for $F$ any
continuous function on $\RR$ with compact support). Here,
$\delta_0\in \ell^2 (\ZZ)$ is just the characteristic function of
$0\in \ZZ$.  This measure $k$ is called the \textit{integrated
density of states}. Let
$$N : \RR\longrightarrow [0,1], N(E):=\int_{(-\infty,E]} d k,$$
be the distribution function of $k$.

\smallskip

There is a direct relation between the measure $k$ and the spectrum
of the $H_\omega$.

\begin{Theorem}\label{theorem-support-and-atomfree}
Let $(\varOmega,T)$ be minimal and uniquely ergodic. Then, the set
$\Sigma$ is the support of the measure $k$. If the function $f$ does
not vanish anywhere then $k$ does not have atoms (i.e. it assigns
the value zero to any set containing only one element).
\end{Theorem}

\begin{Remark}
This is rather standard in the theory of random operators. Specific
variants  of it can be found in many places. In particular, the
first statement of the theorem can be found in \cite{Lenz2}. In the
case of $f$ which do not vanish anywhere the statements of the
theorem are contained in Section 5 of
 \cite{Teschl}. Given the  constancy of the spectrum,  Theorem \ref{Constancy},
 the statements   are also  very special
 cases of the results of Section 5 of
\cite{LPV}.  The key ingredient for the absence of atoms is
amenability of the underlying group $\ZZ$. The statement on the
support does not even need this property.
\end{Remark}

\medskip

As is well known, it is possible to 'calculate' $k$ via an
approximation procedure. This will be discussed next. For $\in \ZZ$
let
$$j_{n} : \ell^2 (\{1,\ldots, n\})\longrightarrow \ell^2 (\ZZ)$$
be the canonical inclusion and let $p_{n}$ be the adjoint of
$j_{n}$. Thus, $$p_{n} : \ell^2 (\ZZ)\longrightarrow \ell^2
(\{1,\ldots, n\})$$ is  the canonical projection. Define for
$\omega\in\varOmega$ then
$$H_\omega^{n} := p_{n} H_\omega j_{n}.$$
We will be concerned with the spectral theory of these operators.
Let the measure $k_\omega^{n}$ on $\RR$ be defined as
$$\int_{\RR} F (x) d k_{\omega}^{n} (x) := \frac{1}{n}
\sum_{k=1}^n \langle F (H_\omega^{n}) \delta_k, \delta_k\rangle$$
(for any  continuous $F$ on $\RR$ with compact support) and let
$$N_\omega^{n} : \RR \longrightarrow [0,1],\: N_\omega^{n}
(E):=\int_{(\infty,E]} d k_\omega^{n} (x),$$ be its distribution
function. Let $E_1,\ldots, E_{n}$ be the eigenvalues of
$H_\omega^{n}$ counted with multiplicity. Then, straightforward
linear algebra (diagonalization of $H_\omega^{n}$ and independence
of the trace of the chosen orthonormal basis) shows that
$$\int_{\RR} F (x) d k_{\omega}^{n} (x) = \frac{1}{n } \sum_{j}
F(E_j)$$ holds for any continuous $F$ on $\RR$ with compact support
and that the distribution functions of the measures $k_{\omega}^{n}
$  are given by
$$N_\omega^{n} (E) = \frac{\#\{ \mbox{Eigenvalues of
$H_\omega^{n}$ not exceeding $E$}\}}{n },$$ where $\sharp$ denotes
the cardinality of a set. In this sense, $k_\omega^{n}$ is just an
averaged eigenvalue counting.

\begin{Theorem}[Convergence of the integrated density of states] \label{theorem-convergence-ids}
Let $(\varOmega,T)$ be minimal and uniquely ergodic. Then, for any
continuous $F$ on $\RR$ with compact support and any $\varepsilon
>0$ there exists an $N\in\NN$ with
$$ \left| \int_\RR F(x) dk (x) - \int_\RR F (x) d k_\omega^{n} (x) \right|
\leq \varepsilon$$ for all $\omega\in\varOmega$ and all $n\in\ZZ$
with $n \geq N$.
\end{Theorem}
\begin{proof} It  is well-known that the  measures
 $(k_\omega^{n})_n$ converge weakly toward $k$ for $n\to\infty$ for almost every
 $\omega\in\varOmega$. This can be found in many places, see
 e.g. Lemma 5.12 in  \cite{Teschl}.
(That lemma assumes that  $f$ does not vanish anywhere but its proof
does not use this assumption.)
  A key step in the proof is the
 use of the Birkhoff ergodic theorem. The desired statement now
 follows by replacing the Birkhoff ergodic theorem with  the uniform
 ergodic theorem (Oxtoby Theorem)  valid for uniquely ergodic systems \cite{Wal}.
\end{proof}

The operators $H_\omega^{n}$ are sometimes thought of as arising out
of the $H_\omega$ by some form of 'Dirichlet boundary condition'.
The previous result is stable under taking different 'boundary
conditions'.  In fact, even a more general statement is true as we
will discuss next. (The more general statement will even save us
from saying what we mean by boundary condition.)  Let for any $n\in
\ZZ$  and $\omega \in \varOmega$ be a selfadjoint operator
$C_\omega^{n}$ on $\ell^2(\{1,\ldots, n\})$ be given. Then, the
statement of the theorem essentially continues to hold if the
operators $H_\omega^{n}$ are replaced by the operators
$$\widetilde{H}_\omega^{n}:=H_\omega^{n} + C_\omega^{n}$$
provided the rank of the $C's$ is not too big. Here, the rank  of an
operator $C$ on a finite dimensional space, denoted by $\rank (C)$,
is just  the dimension of the range of $C$.

In order to be more specific, we  introduce the measure
$\widetilde{k}_\omega^{n}$ on $\RR$ defined as
$$\int_{\RR} F (x) d \widetilde{k}_{\omega}^{n} (x) := \frac{1}{n }
\sum_{k=1}^n \langle F (\widetilde{H}_\omega^{n}) \delta_k,
\delta_k\rangle$$ (for any  continuous $F$ on $\RR$ with compact
support) and its distribution function  given by
$$\widetilde{N}_\omega^{n} : \RR\longrightarrow
[0,1],\; \widetilde{N}_\omega^{n} (E):= \frac{1}{n}
\sharp\{\mbox{Eigenvalues of $\widetilde{H}_\omega^{n}$ not
exceeding $E$} \}.$$

\begin{Corollary}\label{perturbation-ids}
Consider the situation just described. Let $\omega\in\varOmega$ be
given with
$$\frac{1}{n} \rank (C_\omega^{n})\to 0, n \to \infty.$$ Then,
 for
any continuous $F$ on $\RR$ with compact support and any
$\varepsilon
>0$ there exists an $N\in\NN$ with
$$ \left| \int_\RR F(x) dk (x) - \int_\RR F (x) d\widetilde{k}_\omega^{n} (x) \right|
\leq \varepsilon$$ for all  $n \geq N$.
\end{Corollary}
\begin{proof}
A consequence of the min-max principle, see e.g. Theorem 4.3.6 in
\cite{HoJ} shows
$$|\widetilde{N}_\omega^{n} (E)  - N_\omega^{n} (E)|\leq   \frac{1}{n } \rank
(C_\omega^{n})$$ independent of $E\in \RR$. This directly gives
the desired statement.
\end{proof}

The  theorem allows one to obtain an inclusion formula for the
spectrum $\Sigma$. Denote the spectrum of the operator
$\widetilde{H}_\omega^{n}$ by $\widetilde{\Sigma}_\omega^{n}$. By
construction, $\widetilde{\Sigma}_\omega^{n}$ is just the support of
the measure $\widetilde{k}_\omega^{n}$.

\begin{Corollary}\label{Cor-inclusion} Assume the situation of the previous theorem.
Then, the inclusion
$$\Sigma \subset \bigcap_{ n} \overline{ \bigcup_{k\geq
n}\widetilde{\Sigma_\omega^{k}} }$$ holds for all $\omega\in
\varOmega$.
\end{Corollary}

The corollary is somewhat unsatisfactory in that it only gives an
inclusion. In certain cases more is known. This is further discussed
in Section \ref{ids-equal-KvNS-trace}. For a general result on how
to construct approximations whose spectra converge with respect to
the Hausdorff distance we refer the reader to  \cite{BBdN}.

\subsection{The spectrum as a set and the absolute continuity of spectral measures
}\label{spectrum-set-ac} In this section we consider Schr\"odinger
operators associated to locally constant functions on  minimal
subhifts. Here, a function $h$ on a subshift $\varOmega$ over a
finite alphabet is \textit{locally constant} if there exists an
$N>0$ such that the value of $h (\omega)$ only depends on the word
$\omega (-N)\ldots \omega (N)$.  A key distinction in our
considerations will  then be whether $(f,g)$ is periodic or not.

\medskip

The overall structure of the  spectrum in the periodic case is
well-known. This can be found in many references, see e.g. the
monograph \cite{Teschl}.

\begin{Theorem}[Periodic case]\label{Theorem-periodic-abstract}
Let  $(\varOmega,T)$ be  a minimal subshift  and $f,g
:\varOmega\longrightarrow \RR$  locally constant  with
$f(\omega)\neq 0$ for all $\omega\in \varOmega$. If $(f,g)$ is
periodic (with period $N$), then the spectra $\Sigma$ of the
associated Schr\"odinger operators consist of finitely many (and not
more than $N$) closed intervals of positive length and all spectral
measures are absolutely continuous with respect to Lebesgue measure.
\end{Theorem}

\begin{Remark} Note that periodicity of $(f,g)$ may have its
origin in both properties of $(\varOmega,T)$ and properties of
$(f,g)$. For example periodicity always occurs if $f = g = 1$
irrespective of the nature of $(\varOmega,T)$. Also, periodicity
occurs for arbitrary $f,g$  if $(\varOmega,T)$ is periodic (i.e.
there exists a natural number $N\neq 0$ with $T^N \omega = \omega$
for all $\omega\in\varOmega$).
\end{Remark}

\medskip

The previous theorem gives rather complete information on $\Sigma$
in the periodic case. In order to deal with the non-periodic case,
we will need a further assumption on $(\varOmega,T)$. This condition
is linear repetitivity.

\begin{Theorem}[Aperiodic case \cite{BP}] \label{Theorem-aperiodic-abstract}
Let  $(\varOmega,T)$ be  a linearly repetitive subshift  and $f,g
:\varOmega\longrightarrow \RR$  locally constant  with
$f(\omega)\neq 0$ for all $\omega\in \varOmega$. If $(f,g)$ is
non-periodic, then  there exists a Cantor set $\Sigma$ of Lebesgue
measure zero in $\RR$ such that
$$\sigma (H_\omega) = \Sigma$$
for all $\omega \in \varOmega$.
\end{Theorem}

\begin{Remark} The above theorem was first proven in \cite{Lenz}
in the case $f\equiv 1$. This result was then extended in \cite{DL1}
from linearly repetitive subshifts to arbitrary subshifts satisfying
a certain condition known as  Boshernitzan condition (B)  (again for
the case $f \equiv 1$). In the form stated above it can be inferred
from the recent work \cite{BP}, Corollary 4. This corollary treats
the even more general situation, where condition (B) is satisfied.
Condition (B) was introduced by Boshernitzan as a sufficient
condition  for unique ergodicity \cite{Bosh} (see \cite{DL1} for an
alternative approach as well). In our context, we do not actually
need its definition here. It suffices to know that linear
repetitivity implies (B) (see e.g. \cite{DL1}).
\end{Remark}

The previous result deals with the appearance of $\Sigma$ as a set.
It also gives some information on the spectral type.

\begin{Corollary}\label{Corollary-aperiodic-abstract}  Assume the situation of the previous theorem.
Then, no spectral measure  can be absolutely continuous with respect
to the Lebesgue measure.
\end{Corollary}

\subsection{Aperiodic order and discrete  random Schr\"odinger
operators}\label{aperiodic-order-schroedinger} Schr\"odinger
operators with aperiodic order can be considered within the context
of random Schr\"odinger operators. Indeed, they arise in quantum
mechanical treatment of solids.  As this may be revealing we briefly
present this context in this section. Further discussion and
references can be found e.g. in the textbooks  \cite{CFSK,CL}.

\bigskip

Consider a subshift $(\varOmega,T)$ over the finite alphabet
$\mathcal{A}$. Assume without loss of generality that $\mathcal{A}$
is a subset of the real numbers. Let a $T$-invariant probability
measure $\mu$ on $\varOmega$ be given. To these data we can
associate the  family $(H_\omega)_{\omega\in\varOmega}$ of bounded
selfadjoint operators on $\ell^2 (\ZZ)$ acting via
$$ (H_\omega u)  (n) = u (n+1) + u (n-1) + \omega
(n) u (n).$$ Such operators are (slightly special) cases of the
operators considered in the previous section. They arise in the
quantum mechanical treatment of disordered solids. The solid in
question is modeled by the sequence $\omega\in \varOmega$. The
operator $H_\omega$ then describes the behavior of one electron
under the influence of this $\omega$. More specifically, if the
state of the electron is $u_0\in \ell^2 (\ZZ)$ at time $t = 0$ then
the time evolution is governed by the \textit{Schr\"odinger
equation}
$$ (\partial_t u) (t)  = -i H_\omega u (t),\; u(0) = u_0,$$
This equation has a unique solution given by
$$u (t) = e^{- i t H_\omega} u_0.$$
The behavior of this solution is then linked to the spectral
properties of $H_\omega$.

The two basic pieces of 'philosophy' underlying the investigations
are now the following:

\begin{itemize}

\item Increasing  regularity of the spectral measures  increases the
conductance properties of the solid in question.

\item The more disordered the subshift is the more singular the
spectral measures are.

\end{itemize}

Of course, this has to be taken with (more than) a grain of salt. In
particular, precise meaning has to be given to what is meant by
\textit{regularity and singularity of the spectral measures} and
\textit{conductance properties} and  \textit{disorder in the
subshift}. A large part of the theory is then devoted to giving
precise sense to these concepts and then prove (or disprove)
specific formulations of the mentioned two pieces of philosophy.

Regarding the first point of the philosophy we mention
\cite{Gua1,Gua2,GSB,Las} as basic references for proofs of  lower
bounds on transport via quantitative continuity of the spectral
measures.

As for the second point of the philosophy, the  two - in some sense
- most extremal cases are given by periodic subshifts representing
the maximally ordered case on the one hand and the Bernoulli
subshift (with uniform measure) on the other hand representing the
maximally disordered case.

The periodic situation can be thought of as one with maximal order.
As discussed in the previous section the spectral measures are all
absolutely continuous  (hence not at all singular) and the spectrum
consists of non-trivial intervals. These intervals are known in the
physics literature as (conductance) bands.

The Bernoulli subshift can be thought of as having maximal disorder.
In this case all  spectral measures turn out to be pure point
measures and the spectrum  is pure point spectrum with the
eigenvalues densely filling suitable intervals. We refer to the
monographs \cite{CFSK,CL} for details and further references.

The subshifts considered in the previous section are characterized
by some intermediate form of disorder. They are not periodic.
However, they are still minimal and uniquely ergodic and have very
low complexity. So they are close to the periodic situation (or
rather well  approximable by periodic models with bigger and bigger
periods). Accordingly, one can expect the following  spectral
features of the associated Schr\"odinger operators:

\begin{itemize}

\item Absence of absolutely continuous spectral measures (due to the
presence of disorder i.e. the lack of periodicity).

\item Absence of point spectrum (due to the closeness to the
periodic case).

\item Cantor spectrum of Lebesgue measure zero (as a consequence of
approximation by periodic models with bigger and bigger periods and,
hence, more and more gaps).

\end{itemize}

Indeed, a large part of the theory for Schr\"odinger operators with
aperiodic order  is devoted to proving these features (as well as
more subtle properties) for specific classes of models. Further
details and references can be found in the surveys  \cite{Dam,DEG}.
Here, we emphasize that also our discussion of the operators
associated to a certain substitution generated subshift below will
be focused on establishing the above features.

\section{The substitution $\tau$, its finite words $\mbox{Sub}_\tau$ and its subshift
$(\varOmega_\tau,T)$}\label{The-substitution}
 In this section we
study the two-sided subshift $(\varOmega_\tau, T)$  induced by a
particular substitution $\tau$ on  $\mathcal{A} = \{a,x,y,z\}$ with
$$\tau (a) = a x a, \tau (x) = y,\; \tau (y) = z, \;\tau (z) = x.$$
 The one-sided subshift induced by this substitution had
already been studied by Vorobets \cite{Vor1}. Some of our  results
can be seen as providing the two-sided counterparts to his
investigations. The key ingredient in the investigations of
\cite{Vor1} is that the arising one-sided sequences can be
considered as Toeplitz sequences. This is equally true in our case
of two-sided sequences. Thus, it  seems very likely that one could
also base our analysis of the corresponding features on the
connection to Toeplitz sequenes. Here, we will present a different
approach based the  $n$-decomposition and $n$-partition introduced
in \cite{GLN}.

The subshift $(\varOmega_\tau,T)$ will be of crucial importance for
us as it will turn out that the Schr\"odinger operators associated
to it are unitarily equivalent to the Laplacians on the Schreier
graphs of the Grigorchuk's group $G$.

While we do not use it in the sequel we would like to highlight that
the substitution in question has already earlier appeared in the
study of Grigorchuk's group $G$.  Indeed, it  is (a version of) the
substitution used by Lysenok \cite{Lys} for getting a presentation
of Grigorchuk's group $G$. More specifically, \cite{Lys} gives that
$$G = \langle a,b,c,d | 1 = a^2 = b^2 = c^2 = d^2 = \kappa^k ((ad)^4)
 =\kappa^k ((adacac)^4), k = 0,1,2,....\rangle,$$
where $\kappa$ is the substitution on $\{a,b,c,d\}$ obtained from
$\tau$ by replacing $x$ by $c$, $y$ by $b$ and $z$ by $d$.

\subsection{The  substitution  $\tau$ and its subshift: basic features}
Let the alphabet $\mathcal{A} = \{a,x,y,z\}$ be given and let $\tau$
be  the substitution mentioned above  mapping
 $a\mapsto a x a$,
$x\mapsto y$, $y\mapsto z$, $z\mapsto x$. Let $\mbox{Sub}_\tau$ be
the associated set of words given by
$$\mbox{Sub}_\tau =\bigcup_{w\in \mathcal{A}, n\in \NN\cup\{0\} } \mbox{Sub}(\tau^n
(w)).$$ Then, the following three properties obviously  hold true:

\begin{itemize}

\item The letter $a$ is a prefix of $\tau^n (a)$ for any $n\in
\NN\cup\{0\}$.

\item The lengths $|\tau^n (a)|$ converge to $\infty$ for $n\to
\infty$.

\item Any letter of $\mathcal{A}$ occurs in $\tau^n(a)$ for some  $n$.

\end{itemize}

By the first two properties $\tau^n (a)$ is a prefix of $\tau^{n+1}
(a) $ for any $n\in\NN\cup\{0\}$. Thus, there exists a unique
one-sided infinite word $\eta$ such that $\tau^n (a)$ is a prefix of
$\eta$ for any $n\in\NN\cup\{0\}$. This $\eta$ is then  a fixed
point of $\tau$
i.e. $\tau (\eta) = \eta$. We will refer to it as \textit{the fixed
point of the substitution $\tau$}. Clearly, $\eta$ is then a fixed
point of $\tau^n$ as well for any natural number $n$.

By the third property we then have
$$\mbox{Sub}_\tau = \mbox{Sub} (\eta).$$
We can now associate to $\tau$ the subshift
$$\varOmega_\tau:=\{ \omega\in \mathcal{A}^\ZZ : \mbox{Sub} (\omega)\subset
\mbox{Sub}_\tau\}.$$ Note that every other letter of $\eta$ is an
$a$ (as can easily be seen). Thus, $a$ occurs in $\eta$ with bounded
gaps. This implies that any word of $\mbox{Sub}_\tau$ occurs with
bounded gaps (as the word is a subword of $\tau^n (a)$ and $\eta$ is
a fixed point of $\tau^n$). For this reason $(\varOmega_\tau,T)$ is
minimal and $\mbox{Sub}(\omega) = \mbox{Sub}_\tau$ holds for any
$\omega \in \varOmega_\tau$. We can then apply Theorem 1 of
\cite{DL2} to obtain the following.

\begin{Theorem}\label{Theorem-tau-linear-repetitive}
The subshift $(\varOmega_\tau,T)$ is linearly repetitive. In
particular, $(\varOmega_\tau,T)$ is uniquely ergodic and minimal.
\end{Theorem}

\begin{Remark} It is well-known that subshifts associated to
primitive substitutions are linearly repetitive (see e.g.
\cite{DHS,DZ}). Theorem 1 of  \cite{DL2} shows that linear
repetitivity in fact holds for  subshifts associated to any
substitution provided minimality holds. Unique ergodicity is then a
direct consequence of linear repetitivity due to Theorem
\ref{Theorem-unique-ergodicity}.
\end{Remark}

Our further considerations will be based on a more careful study of
the $\tau^n (a)$. We set
$$p^{(0)}:= a\;\:\mbox{and}\;\:
 p^{(n)} :=\tau^n (a)\;\: \mbox{for}\;\: n\in \NN.$$
  A direct calculation gives
$$p^{(n+1)} = \tau^{n+1} (a) = \tau^n (axa) =\tau^n (a) \tau^n
(x)\tau^n (a) = p^{(n)} \tau^n (x) p^{(n)}.$$ Thus, the following
\textit{recursion formula} for the $p^{(n)}$
$$ (RF)\;\:\hspace{1cm}  p^{(n+1)} =  p^{(n)}  s_n p^{(n)}$$ with
$$ s_n = \tau^n (x) =\left\{ \begin{array}{ccc}  x &:& n = 3 k, k\in \NN\cup \{0\} \\ y &:& n = 3k+1, k\in \NN\cup\{0\}\\
z&:& n = 3 k + 2, k \in \NN \cup\{0\}  \end{array} \right. $$ is
valid\label{recursion}.

This recursion formula is a very powerful tool. This will become
clear in the subsequent sections. Here we first note that it implies
$$|p^{(n)}| = 2^{n+1} -1$$
for all $n\geq 0$. We will now   use it to present a formula for the
occurrences of the $x,y,z$ in $\eta$ and to introduce some special
elements in $\varOmega_\tau$.

\begin{Proposition}[Positions of $a,x,y,z$ in $\eta$]
Consider the fixed point  $\eta = \eta_1 \eta_2 \ldots...$ of $\tau
$ on $\mathcal{A}^\ZZ$. Then the following holds.

\begin{itemize}

\item The letter  $a$ occurs  exactly at the positions  $ 1 + 2 k$, $k\in \NN\cup\{0\}$ (i.e. at the odd positions).

\item The letter $x$ occurs  exactly at the positions of the form $2^{3n +1}  + k \cdot
2^{3n +2}$, $n,k\in \NN\cup \{0\}$ (i.e. at the positions of the
form $2^{3n +1} \cdot  m$ with   $m$ an odd integer and $n\in
\NN\cup\{0\}$ arbitrary).
\item The letter $y$ occurs exactly at the positions of the form $2^{3n +2}  + k\cdot
2^{3n +3}$, $n,k\in \NN\cup \{0\}$ (i.e. at the positions of the
form $2^{3n +2} \cdot  m$ with $m$ an odd integer and $n\in
\NN\cup\{0\}$ arbitrary).
\item The letter $z$ occurs exactly  at the positions of the form $2^{3n +3}  + k\cdot
2^{3n + 4}$, $n,k\in \NN\cup \{0\}$ (i.e. at the positions of the
form $2^{3n +3} \cdot m$ with $m$ an odd integer and $n\in
\NN\cup\{0\}$ arbitrary).
\end{itemize}
\end{Proposition}
\begin{proof} We first note that the given sets of positions are
pairwise disjoint and cover $\NN $. Thus, it suffices to show that
the mentioned letters occur at these positions.

The statement for $a$ is clear. The statements for $x,y,z$ can all
be proven similarly. Thus, we only discuss the statement for $x$.
Repeated application of (RF)  shows that
$$\eta = p^{(3n +1)} r_1 p^{(3n +1)}  r_2 p^{(3n +1)} r_3 ...$$
with $r_1,r_2,...\in \{x,y,z\}$.  Moreover, (RF) implies
$$p^{(3n +1)} = p^{(3n)} x p^{(3n)}.$$
Combining these formula we see that $x$ must occur at all positions
of the form
$$ |p^{(3n)}| +1 +  k (|p^{(3n+1)}| +1) = 2^{3n +1} + k \cdot 2^{3n +2}.$$
This finishes the proof.
\end{proof}

\begin{Remark} The previous proposition shows that $\eta$ is a
\textit{Toeplitz sequence} (with periods of the form  $2^l$ for
$l\in \NN$). As mentioned already the analysis of the one-sided
subshift in \cite{Vor1} is based on this property.
\end{Remark}

We now head further to use (RF) to introduce  some special two-sided
sequences. As is not hard to see from (RF), for any $n\in
\NN\cup\{0\}$ and any single letter $s\in\{x,y,z\}$ the word
$p^{(n)} s p^{(n)}$ occurs in $\eta$. Thus, for any $s\in\{x,y,z\}$
there exists a unique element $\omega^{(s)}\in\varOmega_\tau$  such
that
$$\omega^{(s)} = ... p^{(n)} s | p^{(n)}...$$ holds for all natural
numbers $n$, where the $|$ denotes the position of the origin.  The
elements $\omega^{(x)},\omega^{(y)},\omega^{(z)}\in\varOmega_\tau$
will play an important role in our subsequent analysis. They clearly
have the property that they agree on $\NN$ with $\eta$.
 Indeed, they can be shown   to be  exactly those  elements of $\varOmega_\tau$ which agree
with $\eta$ on $\NN$ (see below).  Here, we already note that these
three sequences  are different. Thus, $\varOmega_\tau$ contains
different sequences,  which agree on $\NN$. Hence, $\varOmega$  is
not periodic.

We finish this section by noting a certain reflection invariance of
our system. Recall that a non-empty  word $w = w_1\ldots
w_n\in\mathcal{A}^\ast$ with $w_j\in \mathcal{A}$ is called a
\textit{palindrome} if $w = w_n \ldots w_1$.
 An easy induction using (RF) shows that for any $n\in \NN\cup\{0\}$  the word  $p^{(n)}$ is a palindrome.  It starts and ends with $p^{(k)}$ for any $k\in \NN\cup\{0\}$ with
$k\leq n$. As each $p^{(n)}$ is a palindrome and any word belonging
to $\mbox{Sub}_\tau$ is a subword of some $p^{(n)}$ we immediately
infer that $\mbox{Sub}_\tau$ is closed under reflections in the
sense that the following proposition holds.

\begin{Proposition} \label{Prop-palindrome} Whenever $w = w_1\ldots w_n\in\mathcal{A}^\ast$
with $w_j\in \mathcal{A}$  belongs to $\mbox{Sub}_\tau$ then so does
$\widetilde{w}:= w_n\ldots w_1$.
\end{Proposition}

\subsection{The main ingredient for our further analysis:  $n$-partition and
$n$-decomposition}\label{Section-n-partition} As a direct
consequence of the definitions we obtain that for any
$n\in\NN\cup\{0\}$ the word $\eta$ has a (unique) decomposition as
$$ \eta = p^{(n)} r_1^{(n)} p^{(n)} r_2^{(n)}....$$
with $r_j^{(n)}\in \{x,y,z\}$. Clearly, this decomposition can be
thought of as  a way of writing $\eta$ with 'letters' from the
alphabet $\mathcal{A}_n= \{\tau^n (a), \tau^n (x),\tau^n (y),\tau^n
(z)\} = \{p^{(n)}, x,y,z\}$. Moreover, setting $r_j := r_j^{(0)}$ we
have $r_j^{(n)}= \tau^n (r_j)$ for any $j\in\NN$. This  way of
writing $\eta$  will be called the \textit{$n$-decomposition of
$\eta$}.  It turns out that an analogous  decomposition can actually
be given for any element $\omega\in \varOmega_\tau$. This will be
discussed in this section.

\medskip

Specifically, we will discuss next  that each
$\omega\in\varOmega_\tau$ admits for each $n\in\NN\cup\{0\}$ a
unique decomposition of the form
$$\omega =  ... p^{(n)} s_0 p^{(n)} s_1 p^{(n)} s_2...$$
with
\begin{itemize}
\item $s_k\in\{x,y,z\} $ for all $k\in \ZZ$,
\item the origin $\omega_0$ belongs to $s_0 p^{(n)}$.
\end{itemize}
Such a decomposition will be referred to as
\textit{$n$-decomposition} of $\omega$. A short moment's thought
reveals that if such a decomposition exists at all, then it is
uniquely determined by the position of any of the $s_j$'s in
$\omega$. Moreover, the positions of the $s_j$'s are given by $p +
2^{n+1} \ZZ$ with $p\in \{0,\ldots, 2^{n+1}-1\}$. Thus, the
positions are given by an element of $\ZZ / 2^{n+1} \ZZ$. This
suggests the following definition.

\begin{Definition}[$n$-partition] For  $n\in \NN\cup\{0\}$ we call
an element $P\in  \ZZ / 2^{n+1} \ZZ $ an  $\textit{n-partition}$ of
$\omega\in \varOmega_\tau$ if for any $q\in P$ both
\begin{itemize}
\item $\omega_{q} \in \{x,y,z\}$ and
\item $\omega_{q+1} \ldots \omega_{q+2^{n+1} -1} = p^{(n)}$
\end{itemize}
hold.
\end{Definition}
Clearly, for each $\omega \in\varOmega_\tau$, existence (uniqueness)
of an $n$-partition is equivalent to existence (uniqueness) of an
$n$-decomposition. In this sense these two concepts are equivalent.
It is not apparent that such an $n$-partition exists at all. Here is
our corresponding  result.

\begin{Theorem}[Existence and Uniqueness of $n$-partitions \cite{GLN}]
\label{theorem-n-partition}  Let $n\in\NN\cup\{0\}$ be given. Then
any $\omega\in \varOmega_\tau$ admits a unique $n$-partition $
P^{(n)} (\omega)$ and the map
$$ P^{(n)} : \varOmega_\tau \longrightarrow \ZZ / 2^{n+1} \ZZ, \; \omega \mapsto P^{(n)} (\omega),$$ is
continuous and equivariant (i.e. $P^{(n)} (T\omega) = P^{(n)}
(\omega) + 1$).
\end{Theorem}

Based on n-partitions (and  n-decompositions) and the previous
theorem one can then study the dynamical system $(\varOmega_\tau,T)$
as well as various questions on the structure of $\mbox{Sub}_\tau$.
This is the content of the next sections.

\subsection{The maximal equicontinuous factor of the   dynamical system $(\varOmega_\tau,T)$}
In this section we use n-partitions to study the structure of the
dynamical system $(\varOmega_\tau,T)$.

\bigskip

For any $n\in \NN$ we  can consider  the cyclic group
$\TT^{(n)}:=\ZZ / 2^n \ZZ $ together with the map $A^{(n)}$, called
\textit{addition map}, which sends   $m + 2^n \ZZ $ to $m + 1+ 2^n
\ZZ$. Then, $(\TT^{(n)}, A^{(n)} )$ is a periodic minimal dynamical
system. Moreover, there are natural maps
$$\pi_n : \TT^{(n+1)} \longrightarrow \TT^{(n)}, \;  m + 2^{n+1}
\ZZ\mapsto m + 2^n \ZZ,$$ for any $n\in \NN\cup \{0\}$. These maps
allow one to construct the  topological abelian group $\ZZ_2$ as the
inverse limit of the system $(\TT^{(n+1)},\pi_n)$, $n\in \NN$.
Specifically, the elements of $\ZZ_2$ are sequences $(m_n)$ with
$m_n\in \TT^{(n)}$ and $\pi_n (m_{n+1}) = m_n$ for all $n\in \NN$.
This group is called the \textit{group of dyadic integers}. The
addition maps $A^{(n)}$ are  compatible with the inverse limit and
lift to a  map $A_2$ on $\ZZ_2$ (which is just addition by $1$ on
each member  of the sequence in question). In this way, we obtain a
dynamical system $(\ZZ_2, A_2)$. It is known as the \textit{binary
odometer}. As $A_2$ is just addition one can think of this system as
an 'adding machine'. It is well known that this dynamical system is
minimal.

Now  Theorem \ref{theorem-n-partition} can be rephrased as saying
that  the dynamical system $(\TT^{(n+1)},A^{(n+1)})$ is a factor of
$(\varOmega_\tau,T)$ via the factor map $P^{(n)}$.  Clearly,  the
factor maps $P^{(n)}$ are compatible with the natural canonical
projections $\pi_n$ in the sense that
$$\pi_n \circ P^{(n)} = P^{(n-1)}$$
holds for all $n\geq 1$. Thus, we can 'combine' the  $P^{(n)}$ for
all $n\in\NN$   to get a factor map
$$P_2 : (\varOmega_\tau, T)\longrightarrow (\ZZ_2, A_2),
\omega\mapsto (n\in \NN \mapsto P^{(n-1)}(\omega)).$$

We first use this to   study continuous eigenvalues. Let $(Y,R)$ be
a dynamical system (i.e. $Y$ is a compact space and $R$ is
 a homeomorphism). Denote the unit circle in $\CC$ by $\mathbb{S}^1$.  Then,
$k\in \CC$ is called a \textit{continuous eigenvalue} of the
dynamical system $(Y,R)$ if there exists a continuous not everywhere
vanishing function $f$ with values in $\CC$ on $Y$ satisfying
$$f (R(y)) = k f(y)$$
 for all $y\in Y$. Such a function is  called a
 \textit{continuous eigenfunction}. Then, any
 continuous eigenvalue belongs to $\mathbb{S}^1$ and
the continuous eigenvalues form a group under multiplication
whenever the underlying dynamical system is minimal. Indeed, by
minimality any continuous eigenfunction has  constant
(non-vanishing) modulus. Then, the product of two eigenfunctions is
an eigenfunction to the product of the corresponding eigenvalues.
The complex conjugate of an eigenfunction is an eigenfunction to the
inverse of the corresponding eigenvalue and the constant function is
an eigenfunction to the eigenvalue $1$. Moreover, it is not hard to
see that minimality implies that the multiplicity of each continuous
eigenvalue is one (i.e. for any two continuous eigenfunctions $f,g$
to the same eigenvalue there exists a complex number c with $f = c
g$).

\smallskip

Let now  $\mathcal{E}_{n}$ be the group of continuous eigenvalues of
$\TT^{(n)}$. This is just the subgroup of $\mathbb{S}^1$ given by
$\{e^{2 \pi i  \frac{k}{2^n}} : 0\leq k \leq 2^n-1\}$. Then, clearly
the groups $\mathcal{E}_{n}$ and  $\TT^{(n)}$ are dual to each other
via
$$\mathcal{E}_n\times \TT^{(n)}\longrightarrow \mathbb{S}^1,
(z,m + 2^n \ZZ)\mapsto z^m.$$ The dual maps to the canonical
projections $\pi_n : \TT^{(n+1)}\longrightarrow \TT^{(n)}$ are then
the canonical embeddings
$$\iota_n: \mathcal{E}_n\longrightarrow \mathcal{E}_{n+1}, z \mapsto z.$$
Thus,  $\mathcal{E}_{n}$ is a subgroup of $\mathcal{E}_{n+1}$. Let
$\mathcal{E}$ be the group arising as the union of the
$\mathcal{E}_n$, i.e.
$$\mathcal{E} :=\bigcup_n \mathcal{E}_n.$$
This group is often denoted as $\ZZ(2^\infty)$. Equip it with the
discrete topology and denote its Pontryagin dual
$\widehat{\mathcal{E}}$  by $\TT$.

\begin{Proposition}\label{prop-canonical-isom} The group $\TT$ is canonically isomorphic to the
group $\ZZ_2$ via
$$\ZZ_2\longrightarrow \TT, \; (m_n + 2^{n} \ZZ )_n \mapsto (  z \mapsto z^{m_n},\;  \mbox{whenever $z\in \mathcal{E}$ belongs to $\mathcal{E}_n$}.  )$$
\end{Proposition}
\begin{proof} By construction $\ZZ_2$ comes about as inverse limit of the
$(\TT^{(n+1)},\pi_n)$, $n\geq 1$. Then, the dual group of $\ZZ_2$
arises as the  direct limit of the dual system $(\mathcal{E}_n,
\iota_n)$. This limit is just the group $\mathcal{E}$. Dualising
once more we obtain that $\TT$ is indeed canonically isomorphic to
the dual of $\ZZ_2$. To obtain the actual formula we can proceed as
follows:

Define  $\varepsilon_n:=e^{2 \pi i \frac{1}{2^n}}$. Then, each
$\varepsilon_n$ is a complex  primitive $2^n$-th root of $1$ with
$\varepsilon_{n+1}^2 = \epsilon_n$. Now, consider an arbitrary
element $\gamma \in\TT$ i.e. a character $\gamma : \mathcal{E}
\longrightarrow \mathbb{S}^1$. Then, $\gamma$ is  completely
determined by its values on the $\varepsilon_n$, $n=1,2\ldots$.
Moreover, for each $n$ we have
$$ ( \gamma (\varepsilon_{n+1}) ) ^2 =
\varepsilon_n^{m_n}$$ for a unique  $m_n \in \ZZ / 2^n \ZZ$ as
$$ \left( (\gamma (\varepsilon_{n+1})^2) \right)^{2^n} = (\gamma
(\varepsilon_{n+1}^2) )^{2^n} = (\gamma (\varepsilon_n))^{2^n} =
\gamma (1) =1.$$ It is not hard to see that $m_{n+1}$ goes to $m_n$
under the natural surjection $ \pi_n$. Thus, to each  character
$\gamma : \mathcal{E} \longrightarrow \mathbb{S}^1$ there
corresponds a sequence $(m_1,m_2,\ldots )$ with $m_n \in \ZZ / 2^n
\ZZ$ and  $\pi_n (m_{n+1}) = m_n$  for all $n\in\NN$ (and vice
versa). The set of such sequences is exactly $\ZZ_2$.
\end{proof}

As any eigenvalue belongs to $\mathbb{S}^1$, there is a canonical
embedding of groups $\mathcal{E}\longrightarrow \mathbb{S}^1$. In
fact, as things are set up here this embedding is just inclusion of
subsets of $\CC$.

By duality, this gives rise to a group homomorphism $j :
\ZZ\longrightarrow \TT$  with dense range. This homomorphism induces
then an action $A$ of $\ZZ$ on $\TT$ via
$$ A : \TT\longrightarrow \TT, A \gamma := j(1) \gamma.$$
It is not hard to see that this $A$ corresponds to $A_2$ if $\TT$ is
identified with $\ZZ_2$ according to Proposition
\ref{prop-canonical-isom}.

Disentangling definitions, we also  infer that the action is given
by
$$(A \gamma) (k) = k \; \gamma (k)$$
 (for $\gamma \in \TT$ and $k\in \mathcal{E}$). We denote the
arising dynamical system as $(\TT,A)$. It is isomorphic to the
{binary odometer}.  By its very  construction it is what is called a
\textit{rotation on a compact abelian group} (viz the action $A$
comes about  by multiplication with $j(1)$, where  $j:
\ZZ\longrightarrow \TT$ is a group homomorphism). Thus, by standard
theory (see e.g. \cite{ABKL,BS}) its group of continuous eigenvalues
is exactly given by the dual of $\TT$ i.e. by $\mathcal{E}$.

Now, obviously any eigenfunction of $\TT^{(n)}$ gives immediately
rise to an eigenfunction of $\varOmega_\tau$ for  the same
eigenvalue (by composing with the factor map). As the factor map is
continuous, we obtain in this way continuous eigenfunctions to each
of the  elements from $\mathcal{E}$.  Minimality easily shows that
(up to an overall  scaling) each of these eigenfunctions is unique.
Thus, we obtain a family of continuous eigenfunctions. At this point
it is not clear that all continuous  eigenvalues of
$(\varOmega_\tau,T)$ belongs to $\mathcal{E}$. However, this (and
more)  will be shown later.

\smallskip

We can use the preceding considerations to introduce a closed
equivalence relation $\approx$ on $\varOmega_\tau$ via
$$ \omega\approx \omega{'} :\Longleftrightarrow f(\omega) = f
(\omega{'})\;\:\mbox{for all eigenfunctions corresponding to
eigenvalues from $\mathcal{E}$}.$$ Then, $\varOmega_\tau / \approx$
is a compact topological space when equipped with the quotient
topology.

Clearly, $\omega \approx \omega{'}$ if and only if $T\omega \approx
T\omega{'}$. Thus, the relation $\approx$ is compatible with the
shift operation. Hence, the quotient $\varOmega_\tau / \approx$
becomes a dynamical system with the operation $T^\approx$ induced by
the shift.

Fix now an $\omega_0\in  \varOmega_\tau$. As discussed above
continuous eigenfunctions do not vanish anywhere and the
multiplicity of each continuous eigenvalue is one. Thus, for each
$k\in \mathcal{E}$ there exists a  unique eigenfunction $f_k$ to $k$
on $\varOmega_\tau$ with $f_k (\omega_0) =1$. Then, the arising
system of eigenfunctions will have the property that
$$f_{k_1} f_{k_2} = f_{k_1 + k_2},\: \; f_{-k} =\overline{f_k}$$
for all $k,k_1,k_2\in \mathcal{E}$. Thus, any
$\omega\in\varOmega_\tau$ will give rise to an element of $\TT =
\widehat{\mathcal{E}}$ via
$$ \mathcal{E}\longrightarrow \mathbb{S}^1, k\mapsto f_k (\omega).$$
Even more is true and  the following result holds. It is well-known
and can be found in various places in the literature. Recent
discussions are given in \cite{ABKL,BK,BLM}.

\begin{Lemma}\label{conjugate} The dynamical systems $(\varOmega_\tau/\approx, T^\approx)$
and $(\TT,A)$ are  conjugate via the map
$$[\omega] \mapsto (k\mapsto f_k (\omega)).$$
In particular, the eigenvalues of
$(\varOmega_\tau/\approx,T^\approx)$ are exactly given by the
elements of  $\mathcal{E}$.
\end{Lemma}

We now further investigate $\approx$ and provide a characterization
of $\omega \approx \omega{'}$. Here, we will again use  the special
words $\omega^{(x)},\omega^{(y)},\omega^{(z)}$ introduced above.

\begin{Proposition}[Characterization $\approx$]\label{prop-characterization-approx} For $\omega,\omega{'} \in\varOmega_\tau$ the
relation $\omega \approx \omega{'}$ holds if and only if one of the
following two properties hold:

\begin{itemize}
\item $\omega = \omega{'}$.

\item There exist $s,s{'}\in \{x,y,z\}$ with $s \neq s{'}$ and
$N\in \ZZ$ with $\omega = T^N \omega^{(s)}$ and $\omega{'}  = T^N
\omega^{(s{'})}$.
\end{itemize}
\end{Proposition}
\begin{proof}
 Let $\omega, \omega{'}$ with
$\omega \neq \omega{'}$ and $\omega\approx \omega{'}$ be given. By
definition of $\approx$ and the  above construction of the
eigenfunctions of $(\varOmega_\tau,T)$, we then have that
$$P^{(n)} (\omega ) = P^{(n)} (\omega{'})$$
for all $n\in \NN\cup\{0\}$. Call this quantity $P^{(n)}$. In the
remaining part of the proof we will identify such a $P^{(n)}$ with
its unique representative in $\{0,\ldots, 2^{n+1} -1\}$.

As $\omega \neq \omega{'}$ we infer that one of the sequences
$(P^{(n)})_{n\in \NN\cup\{0\} }$ or $(2^{n+1} - P^{(n)})_{n\in
\NN\cup\{0\} }$
 must be bounded. (Otherwise,
$\omega$ and $\omega{'}$ would agree on larger and larger pieces
around the origin and then had to be equal.) Assume without loss of
generality that $(P^{(n)})$ is bounded. By restricting attention to
a subsequence we can then assume without loss of generality that
$P^{(n)} = P$ for all $n$. After shifting the sequences by  $P$ to the left we
can then assume without loss of generality that $P^{(n)} =0$ for all
$n$. By definition of $P$, there exist then letters $s,s{'}\in
\{x,y,z\}$ with
$$\omega = ... s| p^{(n)}...\;\:\mbox{and}\;\: \omega{'} = ...s{'}
|p^{(n)}...$$ for all $n\geq 0$ , where $|$ denotes the position of the
origin. This gives, by definition of the  $n$-partition that in fact
$$\omega = ...p^{(n)} s| p^{(n)}\;\:\mbox{and}\;\: \omega{'} = ...p^{(n)} s{'}
|p^{(n)}...$$ for all $n\geq 0$. Thus, we obtain $\omega =
\omega^{(s)}$ and $\omega{'} = \omega^{(s{'})}$. As $\omega \neq
\omega{'}$ we infer that $s\neq s{'}$. This finishes the proof.
\end{proof}

The previous result shows that the factor map from $\varOmega_\tau$
to  $\varOmega_\tau/\approx$ is one-to-one except on three  orbits.
This has strong consequences as will be discussed next.

As $(\varOmega_\tau,T)$ is uniquely ergodic, there exists a unique
$T$-invariant probability  measure $\lambda$ on $\varOmega_\tau$.
 The operation $T$  then  induces a unitary operator $U_T $ on the
associated $L^2$-space via
$$U_T : L^2 (\varOmega,\lambda)\longrightarrow L^2
(\varOmega,\lambda),\;  U_T f = f\circ T.$$ An element $f\in L^2
(\varOmega,\lambda)$ (with $f\not\equiv 0$) is called a
\textit{measurable eigenfunction to $k\in \mathbb{S}^1$} if $U_T f =
k f$. The subshift is said to have \textit{pure point spectrum} if
there exists an orthonormal basis of measurable  eigenfunctions.
From the two  previous results we immediately infer the following.

\begin{Theorem} The dynamical system $(\varOmega_\tau,T)$ has pure
point spectrum and any measurable eigenvalue is a continuous
eigenvalue  and belongs to $\mathcal{E}$.
\end{Theorem}
\begin{proof}
The dynamical system $(\TT, A)$ has pure point spectrum with all
eigenvalues being continuous and belonging to $\mathcal{E}$ as it is
a shift on the  compact abelian group $\TT$ which is the Pontryagin
dual of $\mathcal{E}$, \cite{Wal}.  As the  dynamical system
$\varOmega_\tau / \approx$ is conjugate to $(\TT,A)$ due to Lemma
\ref{conjugate} it has also pure point spectrum with all eigenvalues
being continuous and belonging to $\mathcal{E}$.

Now, Proposition \ref{prop-characterization-approx} shows that
factor map from $\varOmega_\tau$ to $\varOmega_\tau/\approx$ is
one-to-one except on three countable orbits. This implies that in
terms of measures the associated $L^2$-spaces are isomorphic. This
easily gives the desired result.
\end{proof}

\begin{Remark} The occurrence of pure point dynamical spectrum is
a key feature in the investigation of aperiodic order. In fact,
while there is no axiomatic framework for aperiodic order a
distinctive feature is (pure) point diffraction. Now, pure point
diffraction has been shown to be equivalent to pure point dynamical
spectrum. For the case of subshifts at hand this can be inferred
(after some work) from \cite{Que}. A more general result (dealing
with uniquely ergodic  Delone systems) was then given in \cite{LMS}.
The result can even further be generalized to measure dynamical
systems and even processes \cite{BL,LM,LS}.
\end{Remark}

The previous theorem  implies that $(\TT,A)$ is exactly the
\textit{maximal equicontinuous factor} of $(\varOmega_\tau,T)$ (see
e.g. \cite{Aus} for definition). Indeed, one of the many equivalent
ways to describe this factor is as the dual group of the group of
continuous eigenvalues.   A recent discussion of this and various
related facts can be found in  \cite{ABKL}. Henceforth, we will
denote the maximal equicontinuous factor of $(\varOmega_\tau,T)$ by
$(\varOmega_\tau^{\max}, T^{\max})$ and the corresponding factor map
by $\pi^{\max}$. Then, our findings so far provide the following
theorem.

\begin{Theorem}[Factor map onto $\varOmega_\tau^{\max}$] \label{Main-max-equicontinuous-factor}
The three dynamical systems $(\varOmega_\tau/\approx$, $T^\approx)$,
$(\TT,A)$ and $(\varOmega_\tau^{\max}, T^{\max})$ are topologically
conjugate. The factor map
$$\pi^{\max} : \varOmega_\tau \longrightarrow \varOmega_\tau^{\max}$$
is one-to-one in all points except on the images of the points of
the orbits of $\omega^{(x)},\omega^{(y)},\omega^{(z)}$. In these
points it is three-to-one.
\end{Theorem}

\begin{Remark}
\begin{itemize}
\item Minimal systems with the property that their factor map to the
maximal equicontinuous factor is one-to-one in at least one point
are known as \textit{almost automorphic systems} (see e.g.
\cite{ABKL} for further details). Their study has attracted a lot of
attention. As the  previous result shows, $(\varOmega_\tau,T)$ is an
almost automorphic system. In fact, as $\varOmega_\tau$ is
uncountable, the factor map is one-to-one in almost every  point
with respect to the unique invariant probability measure $\lambda$
on $\varOmega_\tau$. This is sometimes expressed as  \textit{the
factor map from $(\varOmega_\tau,T)$ to its maximal equicontinuous
factor is almost-everywhere one-to-one}.

\item There is an alternative description of the relation $\approx$ for
almost automorphic systems. More specifically, define the
\textit{proximality relation} $\sim$  by
$$ \omega\sim \omega^{'} \Longleftrightarrow \inf_{n\in\ZZ} d (T^n
\omega,T^n \omega^{'})=0,$$ where $d$ is any metric on
$\varOmega_\tau$ inducing the topology. (Due to compactness of
$\varOmega_\tau$  the relation is indeed independent of the chosen
metric.) Note that the  proximality relation can be though of as
describing asymptotic agreement. Then, for almost automorphic
systems the proximality relation $\sim$ and the relation $\approx$
agree. This can be found in the book of Auslander \cite{Aus}. A
recent discussion is given in \cite{ABKL}.  In fact, this result is
even more general in that one does not need almost automorphy but
only a weaker condition called \textit{coincidence rank one}. We
refrain from further discussion of this concept and refer the reader
to e.g. \cite{BK} for further investigation. We just note that in
our situation  equality of $\sim$ and $\approx$ and the
characterization of $\approx$ in Proposition
\ref{prop-characterization-approx}  imply  that sequences which are
proximal (i.e. asymptotically equal) are in fact equal everywhere up
to one position.

\item In \cite{Vor1} Vorobets shows that the
   one-dimensional subshift associated to
$\tau$ has the binary odometer $(\TT,A)$  as a factor with the
factor map being $1:1$ in all points except three orbits. He uses
this to conclude pure point spectrum and unique ergodicity. Our
corresponding results above for the two - sided subshift can
therefore be seen as analogues  to his results.  However, our
approach is different: It is  based on $n$-partition whereas his
approach is based on Toeplitz sequences.
\end{itemize}
\end{Remark}


\subsection{Powers and the index (critical exponent) of   $\mbox{Sub}_\tau$}
In this section we have a closer look at the structure of
$\mbox{Sub}_\tau$. The main focus will be on occurrences of three
blocks and the index of words (also known as critical exponent). We
will use $n$-partitions in our study in a spirit similar to
\cite{DL5,DL6}.

\bigskip

We start by investigating occurrences of almost four blocks. An easy
inspection of $\eta$ gives the following lemma.

\begin{Lemma}\label{one-three-plus-block} The word $a x a x a x a$ belongs to
$\mbox{Sub}_\tau$.
\end{Lemma}

The previous  result deals with occurrence of a cube (and even a bit
more) of the special word $a x$.    As $\mbox{Sub}_\tau$ is
invariant under $\tau$ this  then yields the occurrence of many more
cubes. This can be used to exclude eigenvalues for Schr\"odinger
operators via the so-called Gordon argument. This is discussed in
\cite{GLN} for the case at hand. Such an application of the Gordon
arguments for  subshifts coming from substitution  goes back to
\cite{Dam0}, see \cite{Dam} for a survey as well.

Here, we turn next to showing that there are no fourth powers
occurring in $\mbox{Sub}_\tau$.

Let us recall from the considerations on $n$-partitions in Section
\ref{Section-n-partition} that there exists a sequence $r_1^{(n)}
r_2^{(n)} ...\in\{x,y,z\}^\NN$ such that the fixed point $\eta$ of
$\tau $ can be written as
$$ \eta = p^{(n)} r_1^{(n)} p^{(n)} r_2^{(n)}....$$
with $r_j^{(n)} =\tau^n (r_j)\in\{x,y,z\}$ for any $n\in\NN\cup
\{0\}$.  This way of writing $\eta$ is  referred to as the
$n$-decomposition of $\eta$. Call the sequence
$$r^{(n)} = r_1^{(n)} r_2^{(n)} \ldots
\in\{x,y,z\}^\NN$$ the \textit{$n$-th derived sequence of $\eta$}.
Note that for any natural number $n$ the combinatorial properties of
the sequence $r^{(n)}$  are exactly the same as the combinatorial
properties  of the sequence $r^{(1)}$ as $\tau^n$  is injective on
$\{x,y,z\}$ and  $r^{(n)} = \tau^{n} ( r^{(1)})$ holds.

\begin{Proposition} \cite{GLN}\label{Isolation} In the derived sequence $r=r^{(1)}$ the  letters $y$ and $z$
always occur  isolated preceded and followed by an $x$. The letter
$x$ always occurs either isolated (i.e. preceded and followed by
elements of $\{y,z\}$) or  in the form $x x x$. In particular, there
is no occurrence of $xxxx$.  The analogue statements hold for any
natural number $n$ for the sequence $r^{(n)}$  (with $x,y,z$
replaced by $\tau^n (x), \tau^n (y)$ and $\tau^n (z)$).
\end{Proposition}

\begin{Remark} In terms of information the derived sequences $r^{(n)}$ are as
useful as the original sequence. We will base our subsequent
investigations on the relatively simple properties of the derived
sequence $r^{(1)}$  given in the preceding proposition.   More
information should  be obtainable from a more detailed study of the
derived sequences.
\end{Remark}


The $n$-decomposition of $\eta$ gives a way of writing $\eta$ as a
concatenation of the words $p^{(n)}$ and elements from $\{x,y,z\}$.
For example $\eta$ can be written as
$$\eta = (a x a) y (a x a) z (a x a) y (a x \underbrace{a ) x (a}_{a x a}  x a) ...
$$
where we have put brackets around the $p^{(1)} = a x a$. Still,
$\eta$ can contain further occurrences of $p^{(1)}$ as indicated in
the preceding formula. More generally, it is not true that  a
$p^{(n)}$ occurring somewhere in $\eta$ is in fact one of the words
$p^{(n)}$ appearing in the $n$-decomposition of $\eta$. However, it
turns out that whenever $p^{(n)} s p^{(n)}$ occurs in $\eta$ then
both of its $p^{(n)}$ actually stem from the $n$-partition. In this
sense, there is some form of alignment. This is the content of the
next proposition \cite{GLN}.

\begin{Proposition}[Alignment of the $p^{(n)} s p^{(n)}$] \label{Prop-Alignment}
Consider a natural number $n$ and $s\in \{x,y,z\}$. If $p^{(n)} s
p^{(n)}$ occurs in $\eta$ at the position $l$ (i.e. $\eta_l
\eta_{l+1} \ldots \eta_{l + | p^{(n)} s p^{(n)}  | -1} = p^{(n)} s
p^{(n)}$ holds), then $l$ is of the form $1 + k 2^{n+1}$ for some
$k\in\NN\cup\{0\}$. This means that if $ p^{(n)} s p^{(n)}$ occurs
somewhere in $\eta$ then both of its words $p^{(n)}$ actually agree
with blocks $p^{(n)}$ appearing in the $n$-decomposition $\eta =
p^{(n)} r_1^{(n)} p^{(n)} r_2^{(n)} p^{(n)} ....$
\end{Proposition}

If $w$ is a finite word in $\mbox{Sub}_\tau$ and $v$ is a prefix of
$w$ and $N$ is a natural number we define the \textit{index of the
word $w$ in $w^N v$} by $N +  \frac{|v|}{|w|}$ and denote it by $Ind
(w, w^N v)$. We then define the \textit{index of the word $w$}  by
$$\mbox{Ind} (w):=\max\{ \mbox{Ind}(w,w^N v) : \mbox{$v$  prefix of $w$}, N\in \NN, w^N
v\in\mbox{Sub}_\tau\}.$$ As our subshift is minimal and aperiodic
the index of every word can easily be seen to be finite.

\begin{Theorem}[Index of $\varOmega_\tau$]
(a) For every $w\in\mbox{Sub}_\tau$ the inequality $\mbox{Ind} (w) <
4 $ holds. In particular, $\eta$ does not contain a fourth power.

(b) We have $ 4 = \sup\{\mbox{Ind} (w) : w\in \mbox{Sub}_\tau\}.$
\end{Theorem}

\textit{Remark.} The supremum over all the indices is sometimes
known as the critical exponent.

\begin{proof} A proof can be found in \cite{GLN}. Here, we only
sketch the idea. By Lemma \ref{one-three-plus-block} the word $w =a
x a x a x a = v^3 a$ (with $v = ax$) belongs to $\mbox{Sub}_\tau$.
For each $n\in\NN$ we then have
$$\tau^n (v^3 a) = (\tau^n (v))^3 \tau^n (a)= p^{(n)} \tau^n (x)  p^{(n)} \tau^n(x)  p^{(n)}  \tau^n (x) p{(n)}$$
and we infer that the index must be at least $4$. Thus, it suffices
to show that $\eta$ does not contain a fourth power. To do so, it
suffices to consider occurrences of powers of words $w$ in $\eta$.
For short words the statement can easily be checked. Consider now
the case $ |p^{(n)}| +1 = 2^{n+1} \leq |w|\leq |p^{(n+1)}|$ for some
$n\geq 1$. Assume that $w w w$ occurs in $\mbox{Sub}_\tau$. Then,
the Proposition  on alignment gives that the length of $w$ is given
by $|w|= |p^{(n)}| +1 = 2^{n+1}$. This then easily implies the
desired statement.
\end{proof}

\begin{Remark} The proof of the theorem shows that the length of any word
$w\in\mbox{Sub}_\tau$ whose cube $w w w$ also belongs to
$\mbox{Sub}_\tau$ is given by $2^n$ for some $n\in\NN$.
\end{Remark}

\subsection{The word complexity of $\mbox{Sub}_\tau$}
In this section we present a result on  the word complexity of
$\mbox{Sub}_\tau$. A detailed proof can be found in \cite{GLN}.

\bigskip

We define the \textit{word complexity} of the subshift
$(\varOmega_\tau,T)$ as  $$\complexity : \NN\cup
\{0\}\longrightarrow \NN\;\: \; \complexity (L) = \mbox{number of
elements of $\mbox{Sub}_\tau$ of length $L$}.$$

Recall that a word $w\in \mbox{Sub}_\tau$ is called \textit{right
special} if the set of its extensions
$$\{s\in \{a,x,y,z\}: w s\in \mbox{Sub}_\tau\}$$
has more than one element.

\begin{Theorem}[Complexity Theorem]
(a) For any $n\geq 2$  and $L = 2^n + k$ with $0\leq k < 2^{n}$ we
have
$$
\complexity (L+1) - \complexity (L)=\left\{ \begin{matrix}  3 :
0\leq k < 2^{n-1} \\ 2 :  2^{n-1} \leq k < 2^n \end{matrix}\right.$$

(b)  The complexity function $\complexity$ satisfies
$$\complexity (1) = 4, \complexity (2) = 6, \complexity (3) = 8$$
and then for any $n\geq 2$ and $L = 2^n + k$ with $0\leq k < 2^{n}$
$$\complexity (L) =\left\{ \begin{matrix}  2^{n+1} + 2^{n-1} + 3 k :   0\leq k < 2^{n-1} \\
 2^{n+1} + 2^n + 2 k  :  2^{n-1} \leq  k < 2^n \end{matrix}\right.$$

(c) Consider  $n\geq 2$  and $L = 2^n + k$ with $0\leq k < 2^{n}$.

\begin{itemize}
\item  If $0\leq k < 2^{n-1}$, then
 there exist  exactly  two right special words of length $L$.  These are given
 by the suffix of $p^{(n)}$ of length $L$ (which can be extended by
 $x,y,z$) and the suffix of $p^{(n-2)} \tau^{n-2} (x)  p^{(n-1)}$ of length $L$ (which can be
 extended by $\tau^{n-2} (x) $ and by $\tau^{n-1} (x)$).

\item  If $2^{n-1}\leq k < 2^{n-1}$, then
 there exists  exactly one  right special words of length $L$.  This is given
 by the suffix of $p^{(n)}$ of length $L$ (which can be extended by
 $x,y,z$).
\end{itemize}

\end{Theorem}
\begin{proof} The proof relies on a detailed investigation of the
$n$-partition of $\eta$. This allows one to directly determine all
words of length $|p^{(n)}| = 2^{n+1} - 1$ in  $\mbox{Sub}_\tau$  and
this gives $\complexity (2^{n+1} - 1)$ for all $n\in \NN\cup \{0\}$.
At the same time the study of the $n$-partition allows one to obtain
a lower bound on the difference $\complexity (L+1) - \complexity (L)
$ for all $L\in\NN$. This in turn gives a lower bound on
$\complexity$. Combining the lower bound and the precise values we
obtain the statements of the theorem.
\end{proof}

\begin{Remark} As $\mbox{Sub}_\tau$ is closed under reflections by
Proposition \ref{Prop-palindrome} the above statements about right
special words easily translate on corresponding statements about
left special words. (Here,  a word $w\in \mbox{Sub}_\tau$ is called
\textit{left special} if the set $\{s\in \{a,x,y,z\}: s w \in
\mbox{Sub}_\tau\}$ has more than one element.) This shows in
particular that the words $p^{(n)}$, $n\in \NN\cup \{0\}$, are both
right special and left special (and are the only words with this
property).
\end{Remark}

\subsection{Generating the fixed point $\eta$ by an automaton} In
this section we present an automaton that  generates the fixed point
$\eta$ of the substitution $\tau$. This is well in line with general
theory on how to exhibit fixed points of substitutions by automata,
see e.g.  the monograph \cite{AS} to which we also refer for
background on automata. There are numerous applications of automatic
sequences in group theory. For a recent example and  a list of
further reference we refer the reader to \cite{GLNS}.

\smallskip

\begin{figure}[h!]
\begin{center}
\hspace*{-1cm}
\includegraphics[scale=0.5]{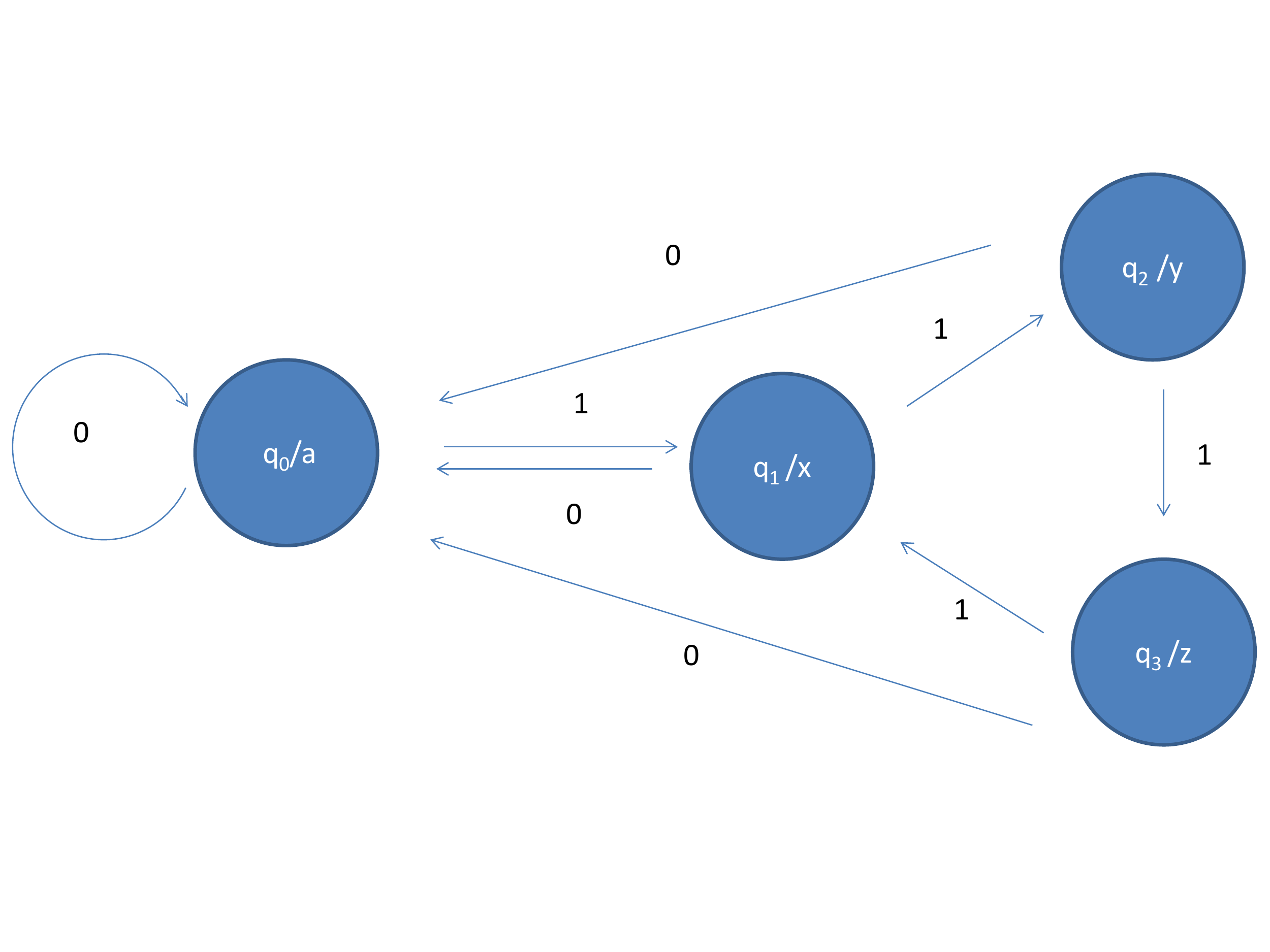}
\hspace*{-1cm}
 \caption{The automaton generating $\eta$}
\label{automaton}
\end{center}
\end{figure}

Consider the automaton $\mathcal{A}$ from figure \ref{automaton}. It
is an automaton over the alphabet $\{0,1\}$ with  four states
$q_0,q_1,q_2,q_3$, which are labeled by $a,x,y,z$ respectively.
Then, the  infinite  sequence
$$\mathcal{A}_{q_0} : \NN\cup\{0\}\longrightarrow \{a,x,y,z\}$$
generated by the automaton with initial state $q_0$ is defined as
follows: Write $n\in \NN\cup \{0\}$ in its binary expansion as
$$n = x_0 2^i + x_1 2^{i-1} + \cdots + x_{i-1} 2 + x_i$$
with $i\in \NN\cup \{0\}$ and $x_j \in \{0,1\}$, $j = 0, \ldots, i$.
Consider now the path $p_n$ in the automaton  starting in $q_0$ and
following the sequence $x_0 x_1\ldots x_i$. Then, $\mathcal{A}_{q_0}
(n)$ is defined to be the label of the state where this path ends.

\begin{Theorem} The fixed point $\eta$ of $\tau$ agrees with
$\mathcal{A}_{q_0}$ (where the fixed point is considered as a map
from $\NN\cup \{0\}$ to $\{a,x,y,z\}$).
\end{Theorem}

This theorem is an immediate consequence of the next proposition. To
state the proposition we will need some further pieces of notation.
For each $n\in \NN\cup \{0\}$ and each state $q$ of the automaton we
define $f^{(n)} (q)$ to be  the   word over $\{a,x,y,z\}$ of length
$2^{n}$ obtained  in the following way: Let $v_1,\ldots, v_{2^n}$ be
the list of  all words of length $n$  over $\{0,1\}$ in
lexicographic order (where $0< 1$). Consider now for each $k =
1,\ldots, 2^n$ the path  in the automaton starting at $q$ and
following the word $v_k$. Then, the $k$-th letter of $f^{(n)}(q)$ is
defined to be the label of the state where this path ends.

Define the letters
$$ s_n = \tau^n (x) =\left\{ \begin{array}{ccc}  x &:& n = 3 k, k\in \NN\cup \{0\} \\ y &:& n = 3k+1, k\in \NN\cup\{0\}\\
z&:& n = 3 k + 2, k \in \NN \cup\{0\}  \end{array} \right.$$ and
recall the recursion formula
$$  p^{(n+1)} =  p^{(n)}  s_n p^{(n)}$$
with  $p^{(n)} :=\tau^n (a)$.

\begin{Proposition} For each natural number $n$ and each $i = 0,1,2,3$ we have
$$ f^{(n+1)} (q_i) = p^{(n)} s_{n+i}.$$
\end{Proposition}
\begin{proof} This is proven by induction. The case  $n=1$ follows
by inspection. Assume now that the statement is true for some $n\geq
1$ and consider $n+1$. Then the lexicographic ordering of the words
of length $n+2$ over $\{0,1\}$ is given by
$$ 0 v_1,\ldots 0 v_{2^{n+1}}, 1 v_1,\ldots 1, v_{2^{n+1}},$$
where $v_1,\ldots, v_{2^{n+1}}$ is the lexicographic ordering of the
words of length $n+1$ over $\{0,1\}$. Thus, from the rules of the
automaton we obtain  for each $i = 0,1,2,3$
$$f^{(n+2)} (q_i) = f^{(n+1)} (q_0) f^{(n+1)} (q_{i+ 1}),$$
where we set $q_{4} = q_1$. From our assumption for $n$ and the
recursion  we then find
$$  f^{(n+2)} (q_i) = p^{(n)} s_{n}  p^{(n)} s_{n + 1+  i} =
p^{(n+1)} s_{n+1+i}$$  for each $i =0,1,2,3$ and this is the desired
statement.
\end{proof}

\subsection{Replacing  $\tau$ by a primitive substitution}
\label{section-recoding}  The substitution $\tau$ arises naturally
in the study of Grigorchuk groups $G$ and its Schreier graphs (see
below). From the point of view of subshifts  it has the (slight)
disadvantage of not being primitive. It turns out that it is
possible to find a primitive substitution $\xi$ with the same fixed
point - and hence the same subshift - as $\tau$.  This can then be
used to obtain alternative proofs for   the (proven above) linear
repetitivity and pure discreteness of the spectrum. This is
discussed at the end of this section. The material presented here
was pointed out to us by Fabien Durand \cite{Dur2}.

\smallskip

Consider the substitution $\zeta$ on the alphabet $\{a,x,y,z\}$ with
$$\zeta (a) = ax, \zeta(x = a y,\zeta (y) = a z, \zeta (z) = ax (= \zeta
(a)).$$ As any letter of the alphabet is contained in $\zeta^4 (s)$
for  any letter $s$, this is a primitive substitution.

To relate it to $\tau$ we  use (again) for $n\in \NN\cup \{0\}$ the
letters $ s_n = \tau^n (x)$ as well  as the recursion formula $
p^{(n+1)} =  p^{(n)}  s_n p^{(n)}$ with  $p^{(n)} :=\tau^n (a)$.

\begin{Proposition} For any natural number $n$  the equality
$$\zeta^n (a) = \tau^{n-1} (a) s_{n-1}$$
holds. In particular, the fixed point $\eta$ of $\tau$ agrees with
the fixed point of $\zeta$.
\end{Proposition}
\begin{proof} This is shown by induction. The cases $n =1,2,3$ are
easily checked by direct inspection. Assume now that the statement
is true for all integers up to  some  $n\geq 3$. Then, we can
compute
$$\zeta^{n+1} (a) = \zeta^{n} (a x) = \zeta^n (a) \zeta^n (x) = \zeta^n (a)
\zeta^{n-1}(ay) = ... = \zeta^n (a) \zeta^{n-1} (a) \zeta^{n-2} (a)
\zeta^{n-3} (ax).$$ By $\zeta (a) = ax $ we then find
$$ \zeta^{n+1} (a) = \zeta^n (a) \zeta^{n-1} (a) \zeta^{n-2} (a) \zeta^{n-2} (a).$$
From our assumption on $n$ we then infer
$$\zeta^{n+1} (a) = \tau^{n-1} (a) s_{n-1} \tau^{n-2} (a) s_{n-2}
\tau^{n-3} (a)  s_{n-3} \tau^{n-3} (a) s_{n-3}.$$ Successive
application of the recursion formula and the fact that $s_{n-3} =
s_n$ then give
\begin{eqnarray*}
\zeta^{n+1} (a) &= & \tau^{n-1} (a) s_{n-1} \tau^{n-2} (a) s_{n-2}
(\tau^{n-3} (a) s_{n-3} \tau^{n-3} (a)) s_{n-3}\\
&=& \tau^{n-1} (a) s_{n-1} \tau^{n-2} (a) s_{n-2} \tau^{n-2} (a)
 s_{n-3}\\
&=& \tau^{n-1} (a) s_{n-1} \tau^{n-1} (a)  s_{n-3}\\
&=& \tau^{n} (a) s_{n-3}\\
&=& \tau^{n} (a) s_{n}.
\end{eqnarray*}
This is the desired statement.
\end{proof}

\begin{Corollary} The subshift $(\varOmega_\zeta,T)$ generated by the
primitive substitution $\zeta$ agrees with the subshift
$(\varOmega_\tau,T)$.
\end{Corollary}

As it is well known, see e.g. \cite{Dur,DZ}, that subshifts
associated to primitive substitutions are linearly repetitive, an
immediate consequence of the previous corollary is that
$(\varOmega_\tau,T)$ is linearly repetitive. Also, as $\zeta$ has
constant length (i.e. the length of $\zeta (t)$ is the same for any
letter $t$) and $\zeta(t)$ starts with $a$ for any letter $t$ we can
apply a result of Dekking \cite{Dek} to obtain purely discrete
spectrum.

\section{Background on graphs and  dynamical
systems}\label{Section-Background-Graphs-etc} In this section we
recall some basic notions and concepts from the theory of graphs,
dynamical systems,  and Schreier graphs. In the next section we will
meet all these abstract concepts in the context of a  particular
group.

Here we first  recall some terminology from the theory of graphs and
introduce the topological space of (isomorphism classes of) rooted
labeled graphs.

Let $\mathcal{B}$ be a finite non-empty set together with an
involution $\mathcal{B}\longrightarrow  \mathcal{B}, b\mapsto
\overline{b}$.  A \textit{graph with edges labeled by $\mathcal{B}$}
is a pair $(V,E)$ consisting of a set $V$ and a set $E\subset
V\times V \times \mathcal{B}$ such that $(v,w,b)$ belongs to $E$ if
 $(w,v,\overline{b})$ belongs to $E$.  The elements of $V$ are
called vertices and the elements of $E$ are called edges. Whenever $
e= (v,w,b)$ is an edge, then $b$ is called the \textit{label},  $v
=o(e)$ the \textit{origin} and $w = t(e)$ the \textit{terminal
vertex} of the edge. We say that \textit{there is an edge from the
vertex
 $v$ the the vertex $w$ with  label $b$} if $(v,w,b)$  belongs to $E$.
 An edge $e$ is said
to \textit{emanate} from the vertex $v$ if $v = o (e)$. The number
of edges emanating from a vertex is called the \textit{degree} of
the vertex. An edge of the form $(v,v,b)$ is called a \textit{loop}
at $v$ (with label $b$).




We will a need the \textit{combinatorial distance} on a graph given
as follows. Each vertex has distance $0$ to itself. The distance
between different vertices   $v$ and $w$ is one if and only if there
exists a label $b$ such that $(v,w,b)$   belongs to $E$. More
generally the distance between different vertices $v$ and $w$ is
then defined inductively as the smallest natural number $n$ such
that there exists a vertex $v'$ with distance $n-1$ to $v$ and
distance one to $w$. If no such $n$ exists the combinatorial
distance is defined to be $\infty$. The graph is called
\textit{connected} if the distance between any two if its vertices
is finite. Likewise the \textit{connected component} of a vertex is
the set of all vertices with finite distance to it.

A \textit{ray} in an infinite graph is an infinite sequence
$v_0,v_1,\ldots$ of pairwise different vertices with distance one
between consecutive vertices. Two rays are equivalent if there
exists a third ray containing  infinitely many vertices of  each of
the rays. An equivalence class of rays is called an \textit{end} of
the graph. For example finite graphs have $0$ ends, the Cayley
graphs of $\ZZ$, $\ZZ^2$ and $\mathbb{F}_2$ (the free group on  two
letters) have $2,1$ and infinitely many ends, respectively. In
general the Cayley graph of a group may have $0,1,2$ or infinitely
many ends.

A \textit{rooted graph} is a pair   consisting of a graph and a
vertex belonging to the vertex set of the graph. This vertex is then
called the \textit{root}.


Two rooted graphs $(G_1, v_1)$ and $(G_2,v_2)$ labeled by the same
set $\mathcal{B}$ are called isomorphic if  there exists a bijective
map $\beta$ from the vertices of $G_1$ to the vertices of $G_2$
taking $v_1$ to $v_2$ such that the vertices $x$ and $y$ in $G_1$
are connected by an edge of color $b$ if and only if their images in
$V(G_2)$ are connected by an edge of color $b$. In this case we
write $(G_1,v_1)\cong (G_2,v_2)$.


Let us now consider the set  $ \mathcal G_* (\mathcal{B})$ of
isomorphism classes of connected rooted graphs labeled with elements
from $\mathcal{B}$ that we endow with the following natural metric.
The distance between the isomorphism classes of the two rooted
graphs $(Y_1,v_1)$ and $(Y_2,v_2)$ is then defined as
$$
\mbox{dist}([(Y_1,v_1)],[(Y_2,v_2)]):=\inf\left\{\frac{1}{r+1} :
B_{Y_1}(v_1,r) \cong B_{Y_2}(v_2,r) \right\}
$$
where $B_Y(v,r)$ is the (labeled) ball of radius $r$ centered in $v$
in the combinatorial metric on $Y$. If we only consider graphs of
uniformly bounded degree (as we will in this paper), the space
$\mathcal G_* (\mathcal{B})$ is compact.

\medskip

We now turn to dynamical systems.  Whenever the group $H$ acts on
the compact space $Y$ via the continuous map $\alpha : H\times
Y\longrightarrow Y$ we call $(Y,H, \alpha)$ a \textit{dynamical
system}. We will mostly suppress the $\alpha$ in the notation. In
particular, we will write $(Y,H)$ instead of $(Y,H,\alpha)$ and we
will write  $t y$ for $\alpha_t (y)$ (with $y\in Y$ and $t\in H$).
If $H$ is the infinite cyclic group $\ZZ$, then the  action of $H$
on $Y$ is determined by $T:=\alpha (1)$ and we then just write
$(Y,T)$ instead of $(Y,\ZZ)$ (and this is well in line with the
notation used in the previous sections).

Whenever  a dynamical system $(Y,H)$ and  $y\in Y$ is given we call
$$\{t y : t\in H\}$$
the \textit{orbit of $y$ (under H)}.

The dynamical system $(Y,H)$ is \textit{minimal} if every orbit is
dense in $Y$. The dynamical system $(Y,H)$ is called
\textit{uniquely ergodic} if there exists  exactly one $H$-invariant
probability measure on $Y$.

The dynamical system $(Y,H)$ is called a \textit{factor} of the
dynamical system $(Y',H)$ if there exist a continuous surjective map
$\chi : Y'\longrightarrow Y$ with $\chi (ty) = t\chi (y)$ for all
$t\in H$ and $y\in Y'$. This map $\chi$ is then called
\textit{factor map}. The dynamical system $(Y,H)$ is then also
referred to as \textit{extension} of the dynamical system $(Y',H)$.

We will deal with graphs arising from dynamical systems. More
specifically, consider a group $H$ generated by a symmetric finite
set $S$ and assume that $H$ acts on the compact space $Y$. Then any
point $y\in Y$ gives rise to the \textit{orbital  Schreier graph of
$y$} denoted by $\varGamma_y$. This is a rooted  graph labeled by
$S$, which is equipped with the involution $S\longrightarrow S,
s\mapsto s^{-1}$. The set of vertices of $\varGamma_y$ is given by
the points in the orbit of $y$. The root is given by $y$ and rhere
is an edge from $v$ to $w$ with label $s\in S$ if $s x = y$. Note
that by the required symmetry of $S$ we then have also an edge from
$w$ to $v$ with label $s^{-1}$ (as is needed according to our
definition of a labeled graph).

As an example we note the Cayley graph of a group $G$ with
generating set $S$. This Cayley graph (with the neutral element of
$G$ as the root)  is the orbital Schreier graph of $G$ corresponding
to the action of $G$ on itself via left multiplication.

If the elements of $S$ happen to all be involutions then there will
be an edge of label  $s$ from $v$ to $w$ if and only if there is an
edge of label $s$ from $w$ to $v$. This is the situation we will
encounter in the next section.

\section{Grigorchuk's group $G$,  its Schreier
graphs and the associated Laplacians}\label{Grigorchuk} In this
section we introduce the main object of our interest: the first
group of intermediate growth  $G$ and the Laplacians on the
associated Schreier graphs.  The group $G$ is  the first group with
intermediate word growth and was introduced by the first author in
\cite{Gri80, Gri84}. By now it is generally known as Grigorchuk's
group $G$ and this is how we will refer to it.\footnote{in spite of
the first author's reluctance} It is generated by four involutions
$a,b,c,d$.
With notation to be introduced presently, the group $G$ can be
viewed as a group acting by automorphisms on the full infinite
binary tree $\mathcal T$. The action extends by continuity to an
action by homeomorphisms on the boundary $\partial \mathcal T$. The
action of $G$ gives rise to Schreier graphs

The Schreier graphs arising from actions of automorphism groups of
infinite regular rooted trees  have attracted substantial attention
in recent years \cite{GNS,Nek}. Of particular interest are so-called
self-similar groups (as $G$)  whose action reflect the self-similar
structure of the tree. Their Schreier graphs also have
self-similarity features. Some are closely related to Julia sets
\cite{BGN,DDMN,Nek}, others are fractal sets close to
 e.g. the Sierpinski gasket or the Apollonian gasket \cite{BG,GS1,GS2}.

The spectra  of the  Laplacians   on such graphs have been described
in some cases \cite{BG,GS1,GS2,GN,GZ}. The spectrum  can be a union
of intervals \cite{GN}, a Cantor set \cite{BG},  or a union of a
Cantor set together with an infinite set of isolated points that
accumulate to it,  as in the case of the so-called Hanoi tower group
\cite{GS2}.  These investigations all use the method introduced in
\cite{BG}. Here we study the spectra of the Laplacians on the
Schreier graphs by a different method via the connection to
aperiodic order. This allows us to determine the spectral type of the Laplacians for arbitrary choice of weights on the generators of the group. We thus obtain one of the rare examples where we understand how the spectrum of the Laplacian depends on the weights. In general this dependence is very poorly understood, as well as the dependence of the spectrum on the generating set in the group. Let us mention here a recent result of Grabowski and Virag who show in \cite{Gra} that there is a weighted Laplacian on the lamplighter group that has singular continous spectrum, whereas the unweighted Laplacian on the same generators is known to have pure point spectrum \cite{GZ}.


\subsection{Grigorchuk's group G}
\label{Group-G}

Let us denote by $\mathcal T_q$, $q\in \NN$ with $q\geq 2$, the
\textit{rooted  regular tree of degree} $q$. The vertex set of
$\mathcal T_q$ is given by $\{0,\ldots, q-1\}^\ast$, i.e.  the set
of all words over the alphabet $\{0,\ldots, q-1\}$. The root of
$\mathcal T_q$ is the empty word. There is an edge between
 $v$ and  $w$ whenever $w = v  k$  or $v = w k$ holds
for some $k\in\{0,\ldots q-1\}$. The words $w \in\{0,\ldots,
q-1\}^n$ constitute the \textit{$n$-the level} of the tree. (In the
tree, they are at combinatorial distance exactly $n$ from the root.)



The  boundary  $\partial \mathcal T_q$ of $\mathcal T_q$ consisting
of infinite geodesic rays in $\mathcal T_q$ emanating from the root
(i.e. infinite paths starting in the root all of whose edges are
pairwise different) can then be identified with the  set
$\{0,1,\ldots, q-1\}^\NN$   of one-sided  infinite words over
$\{0,1,\ldots, q-1\}$.  As mentioned above, the set $\{0,1,\ldots,
q-1\}^\NN$ is equipped with the product topology and is thus a
compact space homeomorphic to the Cantor set.

Any automorphism of $\mathcal T_q$ necessarily preserves the root
(which is the only vertex with degree $q$) and maps paths starting
in the root to paths starting in the root. This readily implies that
any automorphisms group action on $\mathcal T_q$ is level
preserving, i.e. maps words of length $n$ to words of length $n$.
Any such action then extends to an action of the same group by
homeomorphisms on the boundary $\partial \mathcal T_q$.

A regular rooted tree is a self-similar object. Indeed, the subtree
rooted at an arbitrary vertex of the tree is isomorphic to the whole
tree $\mathcal T_q$. The full group of automorphisms inherits this
self-similarity property in the following sense: any automorphism of
$\mathcal T_q$ is completely determined by the permutation it
induces on the $q$ branches growing from the root (an element of
$Sym(q)$) and the collection of $q$ automorphisms
$(g_0,...,g_{q-1})$ which coincide with the restrictions of $g$ on
the corresponding branches.

However, if one is interested in a subgroup $H<Aug(T_q)$ and wants
it to be self-similar, one has to impose the condition that all the
restrictions $(g_0,...,g_{q-1})$ are again elements of the same
group $H$, so that every $g\in H$ can be represented as
\begin{eqnarray*}
g=\alpha(g_0,\ldots,g_{q-1}),
\end{eqnarray*}
where $\alpha$ belongs to $\mbox{Sym}(q)$ and describes the action
of $g$ on the first level of $\mathcal T_q$ and $g_i\in G,
i=0,...,q-1$ is the restriction of $g$ on the full subtree of
$\mathcal T_q$ rooted at the vertex $i$ of the first level of
$\mathcal T_q$. This leads to the following definition.

\begin{Definition}\label{defiselfsimilar}
A group $H$ of automorphisms of $\mathcal T_q$ is {\it self-similar}
if, for all $g\in H, x\in \{0,\ldots, q-1\}$, there exist $h\in H,
y\in \{0,\ldots, q-1\}$ such that
$$
g(xw)=yh(w),
$$
for all finite words $w$ over the alphabet $\{0,\ldots, q-1\}$.
\end{Definition}

We refer the interested reader to \cite{GNS,Nek} for more
information about self-similar groups.


\medskip

We now turn our attention to one particular example of a
self-similar group that will be the central object of our study, the
Grigorchuk group $G$. It is generated by four automorphisms
$a,b,c,d$ of the rooted binary tree $\mathcal T = \mathcal T_2$ as
follows:
\begin{itemize}
\item[] $  a(0 w) = 1 w$,  $ a (1 w) = 0 w;$
\item[] $  b(0w) = 0 a (w)$, $  b (1 w) = 1 c(w);$
\item[] $  c(0w) = 0 a (w)$,  $c (1 w) = 1 d(w);$
\item[] $  d(0w) = 0 w$,  $  d(1w) = 1 b(w),$
\end{itemize}
for an arbitrary word $w$ over $\{0,1\}$. These automorphisms can
also be expressed in the self-similar form, as above:
$$
a=\epsilon(id,id), \qquad b=e(a,c), \qquad  c=e(a,d), \qquad
d=e(id,b),
$$
where $e$ and $\epsilon$ are, respectively, the trivial and the
non-trivial permutations in the group $\mbox{Sym}(2)$.

\begin{Remark}Observe that all the generators are involutions and that
$\{1,b,c,d\}$ commute and constitute a group isomorphic to   the
Klein group $\ZZ /2 \ZZ  \times \ZZ / 2\ZZ$. Let as also mention
that there are many more relations and the group is not finitely
presented.
\end{Remark}

For our subsequent discussion it will be important that $G$ acts
\textit{transitively} on each level, i.e. for arbitrary words $w,u$
over $\{0,1\}$ with the same length there exists a $g\in G$ with $g
u = w$.

\subsection{The Schreier graphs of $G$ and the dynamical system
$(X,G)$}\label{Section-Schreier}


The action of $G$ on the set $V\{0,1\}^\ast$ of vertices of the
rooted binary tree and on its boundary $\partial \mathcal T =
\{0,1\}^\NN$ induces  on this set the structure of a graph labeled
with $\{a,b,c,d\}\subset G$. Specifically, the vertex set of this
graph is given by $\{0,1\}^\ast \cup \{0,1\}^\NN $ and there is an
edge with label $s\in\{a,b,c,d\}$ and  origin  $v$ and terminal
vertex  $w$ if and only if $s v = w$ holds. Note that the set of
arising  edges has indeed the  symmetry property required in our
definition of a labeled graph as any $s\in \{a,b,c,d\}$ is an
involution. For the first three levels of the tree the resulting
graphs are shown in Figure \ref{Schreier-finite}. The connected
components of this graph correspond to the orbits of the group
action. Thus, the  rooted graphs  consisting of such a connected
component together with a root are orbital Schreier graphs (in the
sense defined above).

There are two kind of such connected  components viz finite and
infinite components. The finite ones correspond to the levels of the
tree (recall that the action of $G$ on the levels of the tree is
transitive). We will refer to them simply  as \textit{Schreier
graphs}. The infinite  ones  correspond to the orbits of the action
on the boundary.  The corresponding rooted graphs  will be referred
to as \textit{orbital Schreier graphs}.
\begin{figure}[h!]
\begin{center}
\hspace*{-.5cm}
\includegraphics[scale=0.75]{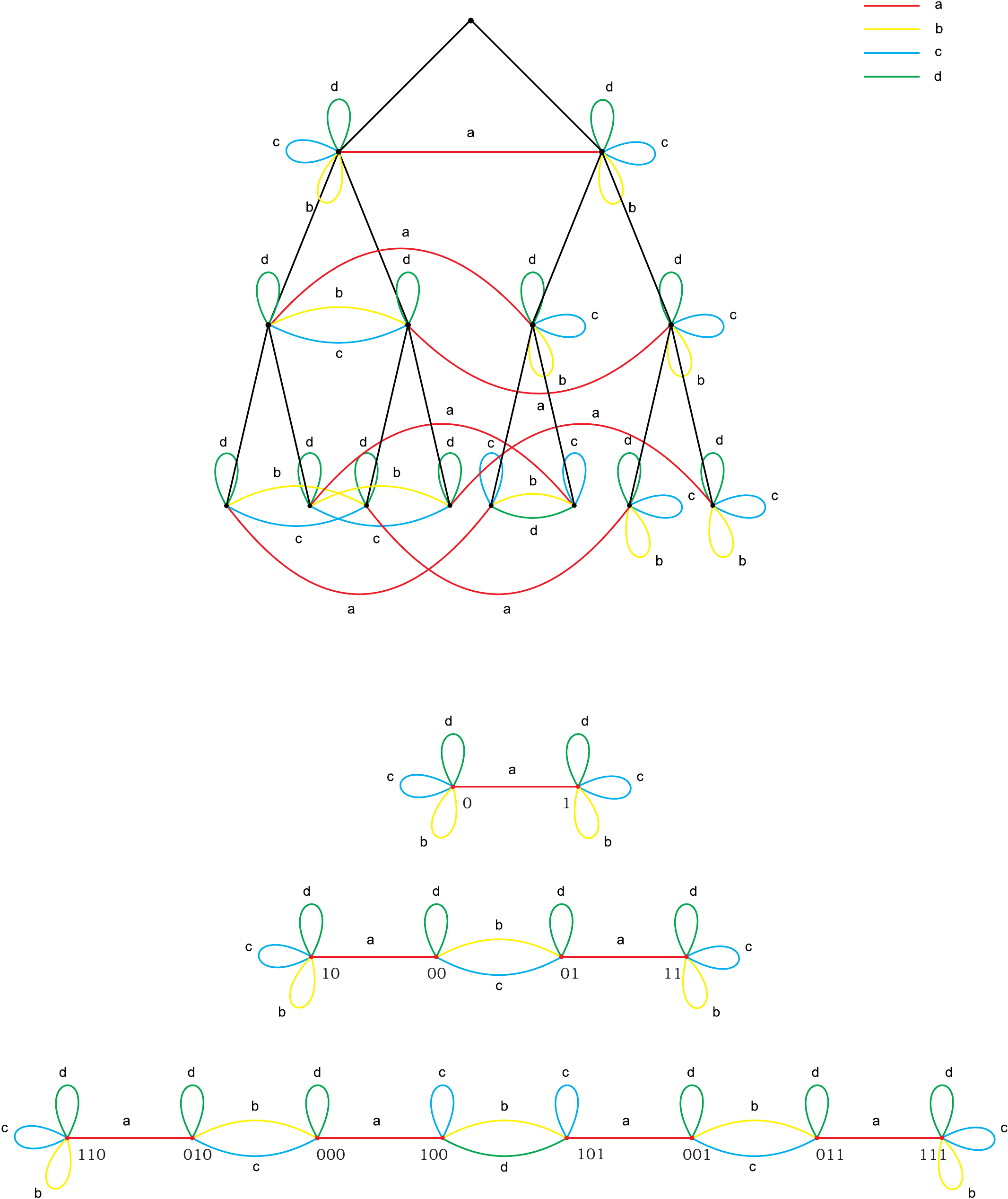}
\caption{The finite Schreier graphs of the first, second and third
level.} \label{Schreier-finite}
\end{center}
\end{figure}



We will need the isomorphism classes of  Schreier graphs and
introduce therefore the map
$$\mathcal{F} :V(\mathcal{T})\cup
\partial \mathcal{T}  \longrightarrow \mathcal G_*
(\{a,b,c,d\}),\;\mathcal F (v):=[ (\varGamma_v, v)],$$ where
$\varGamma_v$ is the connected component of $v$.


The graphs $\varGamma_w$ and $\varGamma_v$ coincide (as non-rooted
graphs) whenever $v$ and $w$ are in the same orbit  of the action of
$G$. In particular, as $G$ acts transitively on each level of the
tree,  for $n\in\NN$ we can therefore define
$$\varGamma_n:= \varGamma_{1^n}$$
which coincides with $\varGamma_w$ for all $w\in V(\mathcal T)$ with
$|w|=n$. In general, the  graph $\Gamma_n$ has  a linear shape; it
has $2^{n-1}$ simple edges, all labeled by $a$, and $2^{n-1}-1$
cycles of length 2 whose edges are labeled by $b,c,d$. It is regular
of degree 4, with one loop at each edge. The loop contributes 1 to
the degree of the vertex because all generators are elements of
order 2, and the labeling of the loop is uniquely determined by the
labeling of the other edges around the vertex, as edges around one
vertex are labeled by $\{a,b,c,d\}$.


The  graphs $\varGamma_\xi$, $\xi \in \partial \mathcal T$
corresponding to the action of $G$  on the boundary are infinite and
have either two ends or one end. The graph $\Gamma_{1^\infty}$
corresponding to the orbit of the rightmost infinite ray, is
one-ended (see Figure \ref{One-ended}), and so are then  all graphs
in the same orbit.

\begin{figure}[ht]
\begin{center}
\hspace*{-.5cm}
\includegraphics[scale=0.75]{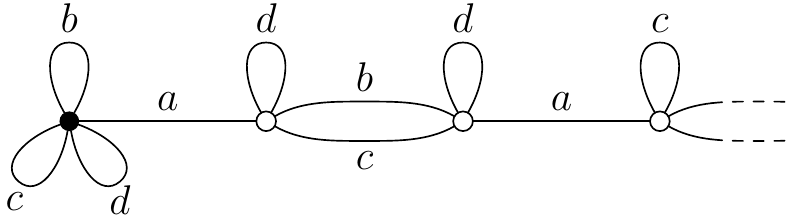} \caption{The one-ended graph $\varGamma_{1^\infty}$.} \label{One-ended}
\end{center}
\end{figure}

All  the other  graphs $\Gamma_\xi, \xi\notin G\cdot 1^\infty$, are
two-ended. They are all isomorphic as unlabeled graphs.

In \cite{Vor2}, Vorobets studied the closure $\overline{\mathcal
F(\partial \mathcal T)}$ of the space of Schreier graphs in the
space of isomorphism classes of rooted labeled graphs. We recall
some of his results next. He showed that the one-ended graphs are
exactly the isolated points of this closure $\overline{\mathcal
F(\partial \mathcal T)}$, and that the other points in
$\overline{\mathcal F(\partial \mathcal T)}$ are two ended graphs.
This suggests to consider the compact set
$$X :=
\overline{\mathcal F(\partial \mathcal T)} \setminus
\{\text{isolated points}\}.$$ The  group $G$ acts on  $X$ by
changing the root of the graph. The arising dynamical system is
denoted as $(X,G)$.  It  is minimal and uniquely ergodic and its
invariant probability measure  will be denoted as $\nu$.

A precise description of $X$  can be given as follows. The space $X$
consists of all isomorphism classes of  two-ended rooted Schreier
graphs $\{(\Gamma_\xi,\xi)  : \xi\in\partial \mathcal T \setminus
G\cdot 1^\infty\}$ and of  three additional countable families of
isomorphism classes  of two-ended graphs. These families  are
obtained by gluing two copies of the one-ended graph $\Gamma_\xi,
\xi\in G\cdot 1^\infty$,  at the root in three possible ways
corresponding to choosing a pair $(b,c)$, $(b,d)$ or $(c,d)$ and
then choosing an arbitrary vertex of the arising graph as the root.
One of these three possibilities is shown in Figure
\ref{Connecting-two-copies}. There, the chosen pair is $(c,d)$ and
to avoid confusion with other edges with the same labels,  labels at
the gluing point are denoted with a prime.


\begin{figure}[ht]
\begin{center}
\hspace*{-.5cm}
\includegraphics[scale=.75]{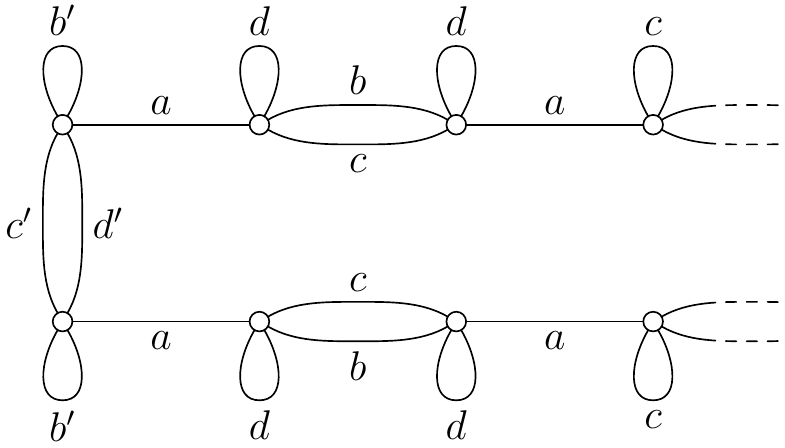}
\caption{Connecting two copies of $\Gamma_{1^\infty}$}
\label{Connecting-two-copies}
\end{center}
\end{figure}

The decomposition of $X$ into isomorphism  classes of the
$(\varGamma_\xi,\xi)$ and the three families mentioned above gives
immediately  rise to a factor map
$$ \phi: X \rightarrow
\partial \mathcal T,$$
which is one-to-one except in a countable set of points, where it is
three-to-one. In fact, the inverse map $\phi^{-1}$ exists on the
complement $\partial \mathcal T\setminus G\cdot 1^\infty $ of the
orbit of the point $1^\infty$ and agrees there with  $\mathcal F$.

Under this factor map $\phi$ the unique $G$-invariant probability
measure $\nu$ on $X$ is mapped to the  uniform Bernoulli measure
$\mu$ on $\partial \mathcal T = \{0,1\}^\NN$.

\subsection{Laplacians associated to the Schreier graphs of
$G$}\label{Operators}


Whenever a group $H$ acts on a measure space $(Y,m)$ by measure
preserving transformations, one
 obtains the \textit{Koopman   representation} $\varrho$  of $H$
on $L^2 (Y,m)$ via
$$\varrho (h) : L^2 (Y,m)\longrightarrow L^2 (Y,m), (\varrho(h) f)(y) = f(h^{-1} y),$$
(for $h\in H$). Any $\varrho (h)$ is then a unitary operator (as the
action is measure preserving).

In our situation when $H=G$,
we have moreover that for $s\in \{a,b,c,d\}$ the unitary operator
$\varrho (s)$ is its own inverse (as $s$ is an involution) and hence
must be selfadjoint. In particular, for any set of parameters
$t,u,v,w\in\RR$ we obtain a selfadjoint operator
$$M_\varrho (t,u,v,w) := t \varrho (a) + u \varrho (b) + v \varrho (c) + w \varrho (d).$$

\smallskip

Consider first an arbitrary  $\xi \in \partial \mathcal T$.  Then,
there is an action of $G$ on the (countable) vertex set of
$V(\varGamma_\xi)$ of $\varGamma_\xi$. Specifically, $s\in
\{a,b,c,d\}$ maps the vertex $x\in V (\varGamma_\xi)$ to the  vertex
$s x$, which is the unique vertex of $V (\varGamma_\xi)$ connected
to $v$ by an edge of label $s$. Clearly, this  action  preserves the
counting measure on $V(\varGamma_\xi)$. Thus, we obtain a
representation $\varrho_\xi$ of $G$ on $\ell^2 (V(\varGamma_\xi))$.
\begin{Definition}[Laplacian of the Schreier graph]  An operator $M_\xi
(t,u,v,w) $ defined by
$$M_\xi (t,u,v,w) := M_{\varrho_\xi} (t,u,v,w) =  t \varrho_\xi (a)  + u  \varrho_\xi (b)  + v \varrho_\xi (c) + w \varrho_\xi (d)$$
with $\xi \in \partial \mathcal T$ and $t,u,v,w\in\RR$ will be
called \textit{(weighted)  Laplacian of the Schreier graph}
$\varGamma_\xi$.
\end{Definition}

\begin{Remark}
\begin{itemize}

\item It is possible to understand $V(\varGamma_\xi)$ as  $G / G_\xi$, where $G_\xi$ is the stabilizer of $\xi$ in the action of $G$ on $\partial \mathcal T$, and then
$\varrho_\xi$ is the quasi-regular representation $\varrho_{G /
G_\xi}$   associated to $G / G_\xi$.

\item If $t,u,v,w$ are positive with $1 = t + u + v + w$ then it is
possible to interpret the operators $M_\xi$ as the Markov operators
of a random walk on the graph $\varGamma_\xi$. In the general case,
the operator $M_\xi (t,u,v,w)$ can still be seen as the natural
weighted  'adjacency matrix'  or 'Laplacian' associated to the the
graph $\varGamma_\xi$.
\end{itemize}
\end{Remark}

\smallskip

We can also equip $\partial \mathcal T =\{0,1\}^\NN$ with the
uniform Bernoulli measure $\mu$ and consider the Koopman
representation $\pi$ of $G$ on $L^2 (\partial \mathcal T,\mu)$ given
via
$$\pi (g) : L^2 (\partial \mathcal T,\mu) \longrightarrow L^2 (\partial \mathcal T,\mu),
\pi (g) f (x) = f(g^{-1} x).$$ This is a unitary representation of
$G$ and  any $\pi (s)$, $s\in \{a,b,c,c\}$, is a unitary selfadjoint
involution. For $t,u,v,w\in \RR$ we then obtain  the operator $M_\pi
(t,u,v,w) $ via
$$M_\pi (t,u,v,w) = t \pi (a) + u \pi (b)  + v \pi (c) + w \pi (d).$$

\smallskip

The following is a crucial result on the spectral theory of the
above operators.

\begin{Theorem}[Independence of spectrum (Bartholdi / Grigorchuk \cite{BG})] \label{Constancy-BG} For any given set of parameters $t,u,v,w\in\RR$ the spectrum
of $M_\xi (t,u,v,w)$ does not depend on $\xi \in \partial \mathcal
T$ and coincides with the spectrum of $M_\pi (t,u,v,w)$.
\end{Theorem}

Of course, there is the  question what the spectrum is in terms of
the parameters $t,u,v,w$.  A complete answer was given in \cite{BG}
in the case $ u = v = w$. The spectrum then consists of two points
or one or two intervals, and an explicit description of the spectrum
can be given in terms of the parameters$ t$ and  $u = v = w$. In
fact, the case $u = v = w$ is the case of periodic Schr\"odinger
type operators and can easily be treated by classical means (Floquet
decomposition). It can also be treated by the  method suggested in
\cite{BG}. For related results on the Kesten-von-Neummann-Serre
trace we refer to \cite{GK}.
Below we will come to a main result of \cite{GLN} which  treats  the
case, where $u = v = w$ does not hold. We will see that in this
case, the spectrum is a Cantor set of Lebesgue measure zero.

\section{The connection between  $(X,G)$ and
$(\varOmega_\tau,T)$}\label{Connecting}

We will now provide a map $\graph$  from $\varOmega_\tau$ to
(isomorphism classes of) labeled rooted graphs with labels belonging
to the alphabet $\{a,b,c,d\}$. In order to get a better
understanding of this map it will be useful to think of the letters
$x,y,z$ as encoding the pairs
$$\left(\begin{matrix} b \\ c \end{matrix}\right) ; \ \  \left(\begin{matrix} b \\ d
\end{matrix}\right) ; \  \left(\begin{matrix} c \\ d \end{matrix}\right)$$
respectively. Roughly speaking the map will convert
$\omega\in\varOmega_\tau$ into a graph with vertex set $\ZZ$ and
root $1$  and edges between $n$ and $n+1$  with labels from
$\{a,b,c,d\}$ according to the value of $\omega (n)$.  We now
provide the  precise \textbf{definition of the labeled rooted graph
$\graph(\omega)$} associated to $\omega \in\varOmega_\tau$  as
follows:

\textit{Vertices:} The set of vertices is $\ZZ$.

\textit{Root:} The number $1$  is chosen as the root.

\textit{Edges:} There are edges
between vertices $n,k$ if an only if $|n-k|\leq 1$. Specifically, edges
are assigned between $n$ and $n+1$ and
from $n$ to itself and from $n+1$ to  itself  in the following way:
\begin{itemize}
\item If $\omega (n) = a$, then there is an edge between $n$ and
$n+1$ and this edge carries the label $a$.

\item If $\omega (n) = x$, then there are two edges between $n$ and
$n+1$; one carries the label $b$ and the other carries the label
$c$. Moreover, there is an additional edge from $n$ to itself
labeled with $d$ and an additional edge from $n+1$ to itself labeled with  $d$.

\item If $\omega (n) =y$, then there are two edges between $n$ and
$n+1$; one carries the label $b$ and the other carries the label
$d$. Moreover, there is an additional edge from $n$ to itself
labeled with $c$ and an additional edge from $n+1$ to itself labeled with  $c$.

\item If $\omega (n) =z$, then there are two edges between $n$ and
$n+1$; one carries the label $c$ and the other carries the label
$d$. Moreover, there is an additional edge from $n$ to itself
labeled with $b$ and an additional edge from $n+1$ to itself labeled with  $b$.

\end{itemize}
We note that in this way any vertex has  for each label
$\{a,b,c,d\}$ exactly one edge of this color emanating from it. We
also note that the arising graphs have a 'linear structure' in a
natural sense.

\smallskip

Let   $ \mathcal G_* (\{a,b,c,d\})$ be the metric space of
isomorphism classes of connected rooted graphs labeled with elements
from $\{a,b,c,d\}$.  Then, $\graph$ gives rise to a map $\Graph$
 via
$$\Graph : \varOmega_\tau \longrightarrow \mathcal G_*(\{a,b,c,d\}), \omega \mapsto [\graph (\omega)],$$
where $[(\varGamma,v)]$ is the equivalence  class of $(\varGamma,v)$
in the space of  isomorphism classes of the rooted connected labeled
graphs.

Recall that $X$ denotes the closure of $\mathcal{F} (\mathcal{T})$
in $\mathcal G_* (\{a,b,c,d\})$  without its isolated points (see
Section \ref{Section-Schreier}). Then, it turns out that the image
of $\varOmega_\tau$ under $\Graph$ is exactly $X$. In fact, much
more is true and this is  our main result on the connection of the
subshift $(\varOmega_\tau,T)$ and the dynamical system $(X,G)$.

Define the maps $A,B,C,D$ from $\varOmega_\tau$ into itself by
\begin{itemize}
\item  $A (\omega) = ...\omega_0 \omega_1 | \omega_2 ...$ if $\omega_1 =a$ and $ A (\omega) = ....\omega_{-1}| \omega_0 \omega_1...$ if $\omega_0 =a$.
\item $ B (\omega) = ...\omega_0 \omega_1 |\omega_2...$ if $\omega_1 \in\{x,y\}$, $B  (\omega) = ....\omega_{-1} |\omega_0 \omega_1...$ if $\omega_0 \in \{x,y\}$ and $B (\omega) = \omega$ in all other cases.
\item $C(\omega)  =... \omega_0 \omega_1 |\omega_2...$ if $\omega_1 \in\{x,z\}$, $C (\omega) = ....\omega_{-1} |\omega_0 \omega_1...$ if $\omega_0 \in \{x,z\}$ and $C(\omega) = \omega$ in all other cases.
\item $D  (\omega) = ...\omega_0 \omega_1 |\omega_2...$ if $\omega_1 \in\{y,z\}$, $D (\omega) = ....\omega_{-1} |\omega_0 \omega_1...$ if $\omega_0 \in \{y,z\}$ and $D (\omega) = \omega$ in all other cases.
\end{itemize}
Then, clearly, $A,B,C,D$ are homeomorphisms and involutions. Let $H$
be the group generated by $A,B,C,D$ (within  the group of
homeomorphisms of $\varOmega_\tau$).

\begin{Theorem}[Factor theorem \cite{GLN}]
\label{Main-connection} The following statements hold:

(a) The group $G$ is isomorphic to the group $H$ via  $\varrho :
G\longrightarrow H$ with  $\varrho(a) =A$, $\varrho(b) = B$,
$\varrho(c) =C$ and $\varrho (d) =  D$.   In particular, there is a
well defined action $\alpha$ of $G$ on $\varOmega_\tau$ given by
$\alpha_g (\omega) :=\varrho (g) (\omega)$ for $g\in G$ and $\omega
\in \varOmega_\tau$ and via this action we obtain a dynamical system
$(\varOmega_\tau,G)$.

(b) The dynamical system $(X,G)$ is a factor of the dynamical system
$(\varOmega_\tau,G)$   with factor map
$$\psi : \varOmega_\tau \longrightarrow X, \omega \mapsto \Graph
(\omega),$$ which is two-to-one.

(c) For every $\omega \in \varOmega_\tau$ the orbits $\{T^n \omega :
n\in \ZZ\}$ and $\{\alpha_g (\omega) : g\in G\}$ coincide.

(d) The dynamical system $(\varOmega_\tau,G)$ is uniquely ergodic
and  the unique $T$-invariant probability measure on
$\varOmega_\tau$ coincides with the unique  $G$-invariant
probability measure on $\varOmega_\tau$.
\end{Theorem}

The theorem provides a factor map $\psi :
(\varOmega_\tau,G)\longrightarrow (X,G)$. In Section
\ref{Section-Schreier} we have already  encountered the factor map
$\phi : (X,G)\longrightarrow (\partial \mathcal{T}, G)$. Putting
together these factor maps  provide a tower of extensions of the
dynamical system $(\partial \mathcal{T}, G)$ via
$$(\varOmega_\tau,G)\stackrel{\psi}{\longrightarrow} (X,G)
\stackrel{\phi}{\longrightarrow}(\partial \mathcal{T}, G).$$ The
remainder of this section is devoted to providing additional
perspective on these two factor maps by discussing their respective
merits. We will see that $\phi$ resolves the non-typical points and
that $\psi$ allows one  to embed the group $G$ into a topological full group.

We will need a bit of notation first. Let  $(Y,H)$ be a dynamical
system. Let $g\in H$ be given. Then, an $y\in Y$ is called
\textit{$g$-typical} if either $g y \neq y$ or $g$ acts trivially in
some neighborhood of $y$. If $y\in Y$ is $g$-typical for any $g\in
H$ it is called \textit{typical}. Typical points have many claims to
be indeed typical. For instance, if $H$ is countable the set of
typical points has a meager complement \cite{Gri11}. Moreover, the
following is shown in \cite{GNS,Gri11}.

\begin{Proposition} Let $(Y,H)$ be a minimal dynamical system and $H$
finitely generated. Then, the Schreier graphs of typical points are
locally isomorphic (as labeled graphs). Moreover, if $x\in Y$ is
typical and $y\in Y$ is not typical, then  any ball around a vertex
of $\varGamma_x$ is isomorphic to a ball around some vertex in
$\varGamma_y$.
\end{Proposition}

In order to get a further understanding of the typical points we
will next discuss another characterization of typical points. Define
the \textit{stabilizer} of $y$ as
$$\mbox{st}_H (y) :=\{ h\in H : h y = y\}.$$
and the \textit{neighborhood stabilizer} of $y$ as
$$\mbox{st}_H^0 :=\{ g\in H : \mbox{ $h$ acts as identity on a
neighborhood of $y$}\}.$$ Then, $\mbox{st}_H^0$ is a normal subgroup
of $\mbox{st}_H$ and the quotient
$$\mbox{germ} (y):=\mbox{st}_H (y) /\mbox{st}_H^0 $$
is called the \textit{group of germs} (at $y$).

\begin{Lemma} Let $(Y,H)$ be a dynamical system. Then, the following
assertions are equivalent for $y\in Y$.
\begin{itemize}
\item[(i)] The point $y$ is  typical.
\item[(ii)] The group of germs $\mbox{germ} (y)$ is trivial.
\end{itemize}
\end{Lemma}
\begin{proof}(i)$\Longrightarrow$ (ii): If $\mbox{germ}(y)$ is not
trivial, there exits a $g\in\mbox{st}_H \setminus\mbox{st}_H^0$.
Hence, $y$ is not typical.

\smallskip

(ii)$\Longrightarrow$ (i): If $y$ is not typical then there exists a
$g\in G$ with $g y = y$ but $g$ is not the identity on a
neighborhood of $y$. Thus, $\mbox{germ} (y)$ is not trivial (as it
contains the class of $g$).
\end{proof}

\begin{Remark} The points $y$ with non - trivial $\mbox{germ}(y)$
are sometimes called \textit{singular}.
\end{Remark}

We will also need the concept, going back to \cite{GPS},  of the
(topological) \textit{full group} $[[T]]$ of a homeomorphism $T :
Y\longrightarrow Y$ of the Cantor set $Y$. This is the group of
those homeomorphisms of $Y$ that at any point coincide locally with
powers of T.  It is a   countable group with remarkable properties.
For example it is amenable if $T$ is minimal \cite{JM}, its
commutator subgroup is simple  \cite{Matui} and it is finitely
generated if and only if $T$ is a minimal subshift over a finite
alphabet \cite{Matui}.

After these preparations we can now came back to the situation at
hand.

The dynamical system $(\partial \mathcal{T}, G)$  is minimal and
uniquely ergodic. However, it does have non-typical points. In fact,
a point in $\partial \mathcal{T}$ is non-typical if and only if it
belongs to  $G\cdot 1^\infty$ (see \cite{Vor2}).

The dynamical system  $(X,G)$   is minimal and uniquely ergodic (as
discussed in \cite{Vor1}). Moreover, all its points are typical (as
can easily be seen). Thus, from this point of view  a key merit of
$(X,G)$ is to provide an extension of of $(\partial \mathcal{T}, G)$
which resolves the non-typical points.

The dynamical system $(\varOmega_\tau,G)$  is minimal and uniquely
ergodic and all its points are typical (as can easily be seen).
Moreover, as shown in the previous theorem, the dynamical system
$(\varOmega_\tau,G)$ has the additional feature that there exists a
homeomorphism $T$ on $\varOmega_\tau$ such that the orbits of $T$
agree with the orbits of $G$. This in turn can  be seen to imply
that $G$ is a subgroup of the topological full group $[[ T]]$ of $T$. In fact,
$G$  embeds  into the topological full group of $(\Omega_\tau,T)$,
as  the action   of generators   $a,b,c,d$ on $\Omega_\tau$ can  be
represented locally  as the  action by $T^{\pm 1}$ and $T^0 = id$,
so  $G$ embeds  into $[[T]]$.

\begin{Remark} \label{Connection-Bon}
In this context we also mention a recent article of Matte Bon
\cite{MBo} showing that the group $G$ (and other groups of
intermediate growth introduced by the first author in \cite{Gri80})
embed into the topological full group $[[\phi]]$ of a minimal
subshift $\phi$ over a finite alphabet. While his approach is
different from ours  it leads to the the same
subshift for the group $G$.
\end{Remark}

We finish this section with the question whether there exists a
minimal homeomorphism $S$ on $X$ such that $G$ is a subgroup of
$[[S]]$.

\section{Spectral theory of the $M_\xi$ for  $\xi \in \partial \mathcal{T}$
}\label{Section-Spectral} In this section  we discuss spectral
theory of the operators $M_\xi (t,u,v,w)$  (defined in Section
\ref{Operators}) for arbitrary values of the parameters $t,u,v,w$.
We do so by using the results of the previous section to  transfer
the problem of spectral theory of the operators $M_\xi (t,u,v,w)$,
$\xi \in
\partial \mathcal T$, to the field of the spectral theory of
discrete Schr\"odinger operators with aperiodic order,
$(H_\omega)_{\omega\in\varOmega_\tau}$.

\bigskip

We consider the subshift $(\varOmega_\tau,T)$. In order to define
the Schr\"odinger operators  we will define specific functions $f,g$
on $\varOmega_\tau$ depending on four real parameters $t,u,v,w$.
Given these parameters we set

$$D := u + v + w$$
and define
$$ f: \varOmega_\tau\longrightarrow \RR\;\mbox{by}\;\: f(\omega) :=
\left\{ \begin{array}{ccc}  t  &:& \omega_0 = a \\   D - w &:&
\omega_0 = x\\ D - v &:& \omega_0 = y\\ D- u & : &  \omega_0 =
z\end{array} \right.$$ and
$$g : \varOmega_\tau \longrightarrow \RR\;\: \mbox{by}\;\: g(\omega) :=
\left\{ \begin{array}{c}  w : \omega_0 \omega_1 \in\{ ax, xa\} \\ v
:
\omega_0 \omega_1 \in \{ay,ya\}\\
u : \omega_0 \omega_1 \in  \{az,za\}
\end{array} \right..$$

For a given $(t,u,v,w)\in \RR^4$ and these  $f,g$  we let
$H_\omega$, $\omega \in\varOmega_\tau$, be the associated operators.

\begin{Proposition}\cite{GLN}\label{unitary-equivalence}
Let $(t,u,v,w) \in \RR^4$ be given. Let $\xi \in \partial \mathcal T
\setminus  G \cdot 1^\infty$ be arbitrary. Then, there exists an
$\omega$ in $\varOmega_\tau$  such that $H_\omega$ is unitarily
equivalent to $M_\xi (t,u,v,w)$.
\end{Proposition}

As a consequence of the previous proposition we can translate
results on the spectral theory of Schr\"odinger operators associated
to $(\varOmega_\tau,T)$ into results on the $M_\xi$. Recall that the
spectrum of the $M_\xi (t,u,v,w)$ does not depend on $\xi\in
\partial \mathcal T$ (due to Theorem \ref{Constancy-BG}). Define
$$\mathcal{P}:=\{(t,u,v,w) \in \RR^4 : t\neq 0, u+ v \neq 0, u + w
\neq 0, v + w \neq 0\}.$$

\begin{Theorem}[Intervals vs Cantor spectrum for the $M_\xi$ \cite{GLN}] \label{Main-Laplacian-Cantor}
 Let $(t,u,v,w)\in\mathcal{P}$ be given
and let $\Sigma = \Sigma (t,u,v,w)$ be the spectrum of the
associated family of Laplacians $M_\xi (t,u,v,w)$, $\xi \in \partial
\mathcal T \setminus G \cdot 1^\infty$. Then, the following holds:

(a) If $u =v = w$ holds then $\Sigma$ consists of one or two closed
non-trivial intervals and all spectral measures are absolutely
continuous.

(b) If $u = v = w$ does not hold then $\Sigma$ is a Cantor set of
Lebesgue measure zero and no spectral measure is absolutely
continuous.
\end{Theorem}
\begin{proof} Note that
$(\varOmega_\tau,T)$ is linearly repetitive due to Theorem
\ref{Theorem-tau-linear-repetitive}.  Now, the  theorem is a direct
consequence of the preceding proposition and Theorem
\ref{Theorem-periodic-abstract} and Theorem
\ref{Theorem-aperiodic-abstract}.
\end{proof}

\begin{Remark} The case $u = v = w$ has already be treated in \cite{BG} and an
explicit description of the spectrum (in terms of the value of $u$)
has been given there.
\end{Remark}

We can also use the above considerations to translate results  on
absence of  eigenvalues from the $(H_\omega)$ to the $M_\xi$. Recall
that $\mu$ denotes the uniform Bernoulli measure on $\partial
\mathcal T = \{0,1\}^\NN$.

\begin{Theorem}[Absence of eigenvalues \cite{GLN}] \label{Main-Laplacian-Absence-eigenvalues} Let $(t,u,v,w)\in\mathcal{P}$ be given
and assume that $u = v = w$ does not hold.

(a) For $\mu$-almost every $\xi \in \partial \mathcal T$ the
operator $M_\xi$ does not have eigenvalues.

(b) For any $x\in \phi^{-1} (G 1^\infty)$ the  operator $M_x$ does
not have eigenvalues.
\end{Theorem}

\begin{Remark} The considerations of this section are concerned with the case $(t,u,v,w)\in \mathcal{P}$.
For  $(t,u,v,w)\notin\mathcal{P}$ the operators in question can be
decomposed as a sum of finitely many different finite dimensional
operators each appearing with infinite multiplicity. Thus, the
spectrum is pure point with finitely many eigenvalues each with
infinite multiplicity.
\end{Remark}

\section{Integrated density of states  and  Kesten-von-Neumann-Serre spectral measure} \label{IDS-Schreier}
In this section we  show that the Kesten-von-Neumann-Serre spectral
measure actually agrees with the integrated density of states.  This
result is folklore and certainly known to experts. Also, a recent
discussion of a somewhat more general setting can also be found in
\cite{ATV}. Here, we give a direct reasoning for the case at hand
based on our approach.

\bigskip

There is one more representation studied in the context of spectral
approximation. This representation comes from the action of $G$ on
the $n$-th level of the tree. This action clearly preserves the
counting measure on the finite set of vertices of the $n$-th level.
It hence gives rise to a representation $\varrho_n$ of $G$ on the
finite dimensional
 $\ell^2 (V(\varGamma_n))$. For $t,u,v,w\in \RR$ we then set
 $$M_n (t,u,v,w) := M_{\varrho_n} (t,u,v,w) =t \varrho_n (a)  + u  \varrho_n (b)  + v \varrho_n (c) + w \varrho_n (d).$$

To each such $M_n (t,u,v,w)$ we associate the \textit{spectral
distribution} which is the measure $\mu_n (t,u,v,w)$ on $\RR$ given
by
$$\mu_n (t,u,v,w) :=\frac{1}{|V(\varGamma_n)|} \sum_{E} \delta_E,$$
where the sum runs over  eigenvalues $E$ of $M_n (t,u,v,w)$ counted
with multiplicities. In the case $ u = v = w$ it  is shown in
\cite{GZ3} that the measures $\mu_n$, $n\in\NN$, converge weakly and
the limiting measure
$$\mu_\infty(t,u,v,w) = \lim_{n\to \infty}  \mu_n (t,u,v,w)$$
is called the \textit{Kesten-von-Neumann-Serre spectral measure}
there. Next we will show that the limiting measure exists for any
values of the parameters $t,u,v,w$ and coincides with the so-called
\textit{integrated density of states} of the associated
Schr\"odinger operators.

For given $(t,u,v,w)\in\RR^4$ we chose  the functions $f,g$ as in
the previous sections and let $H_\omega$, $\omega\in\varOmega_\tau$,
be the associated Schr\"odinger operators.

\begin{Theorem}\label{ids-equal-KvNS-trace} Let $(t,u,v,w)\in\RR^4$ be given.
For $n\in\NN$, let $M_n (t,u,v,w)$ be  the corresponding operator on
$\varGamma_n$ and $\mu_n$ its spectral distribution. Then, the
sequence of measures  $(\mu_n (t,u,v,w))_n$ converges weakly towards
the integrated density of states $k$ of $(H_\omega)$. In particular,
the integrated density of states of the $(H_\omega)$ agrees with the
Kesten-von-Neumann-Serre spectral measure.
\end{Theorem}
\begin{proof}
The graph arising from restricting $\graph (\omega^{(x)})$ to the
vertex set $[1,|\tau^n (a)|]$ and the graph $\varGamma_n$ differ at
most in $6$ loops at the ends (as is clear from the definitions).
Thus, a simple variant of the argument in Proposition
\ref{unitary-equivalence} shows  that the restriction of the
operator $H_{\omega^{(x)}}$ to $[1,|\tau^n (a)|]$ is a perturbation
of $M_n (t,u,v,w)$  of rank at most $6$. Now, Corollary
\ref{perturbation-ids} gives the convergence of the  $\mu_n
(t,u,v,w)$ towards the integrated density of states of $(H_\omega)$.
\end{proof}

A few comments on this result are in order. They are gathered in the
next remark.

\begin{Remark}
\begin{itemize}

\item Denote the $\xi$-independent spectrum of the operators $M_\xi
(t,u,v,w)$ by $\Sigma (t,u,v,w)$  and the spectrum of $M_n
(t,u,v,w)$ by $\Sigma_n (t,u,v,w)$. Clearly,  the support of $\mu_n$
is given by $\Sigma_n (t,u,v,w)$. By Theorem
\ref{theorem-support-and-atomfree} support of the integrated density
of states $k$ is given by $\Sigma (t,u,v,w)$. Thus, the previous
theorem gives in a certain and very weak sense the convergence of
the $\Sigma_n (t,u,v,w)$ towards $\Sigma (t,u,v,w)$. In particular,
there is an inclusion of spectra as shown in   Corollary
\ref{Cor-inclusion}.

In the case at hand, however, convergence of the $\Sigma_n
(t,u,v,w)$ towards $\Sigma (t,u,v,w)$ holds in a much stronger
sense. More precisely, the results of \cite{BG} give the following:
\begin{itemize}
\item[(a)] $\Sigma_n (t,u,v,w)\subset \Sigma_{n+1} (t,u,v,w)$ for all $n\in\NN$.
\item[(b)] $\Sigma (t,u,v,w) =\overline{\cup_n \Sigma_n (t,u,v,w)}$.
\end{itemize}
These are rather remarkable features and not at all true for
approximation of the integrated density of states of Schr\"odinger
operators in other cases.

Let us note, however, that there are recent results on approximation
of spectra of minimal subshifts  with respect to the hausdorff
distance by suitable periodic approximations \cite{BBdN}.

\item It is also worth pointing out that the approximation of the
spectra is done 'from below' i.e. via unions. This is in contrast to
other approximation schemes used in the investigations of
Schr\"odinger operators, where the approximation is done 'from
above' i.e. via intersections.

\item The feature presented in the preceding two points also explain
that the spectrum of $M_\pi (t,u,v,w)$ agrees with $\Sigma
(t,u,v,w)$. The reason is simply  that the representation $\pi$
decomposes as a sum of the representations $\pi_n$ (as shown in
\cite{BG}). Accordingly, $M_\pi (t,u,v,w)$ is a sum of the finite
dimensional operators $M_n (t,u,v,w)$, $n\in \NN$, and its spectrum
is then given as the closure of the union of the spectra of the $M_n
(t,u,v,w)$.

\end{itemize}
\end{Remark}


\begin{thebibliography}{99}

\bibitem{ATV} M. Abert, B. Virag, A. Thom:
\textit{Benjamini-Schramm convergence and pointwise convergence of
the spectral measure  }, Preprint 2013.

\bibitem{AS} J.-P. Allouche, J.~Shallit, \textit{
Automatic Sequences: Theory, Applications, Generalization},
Cambridge University Press, Cambridge (2003).


\bibitem{Aus}  J. Auslander:  \emph{Minimal Flows and their Extensions},
North-Holland Mathematical Studies 153, Elsevier 1988.

\bibitem{ABKL} J.-B. Aujogue, M. Barge, J.  Kellendonk, D. Lenz: \textit{Equicontinuous factors, proximality and Ellis
semigroup for  Delone sets}, to appear in \cite{KLS}.

\bibitem{BGr} M.~Baake, U.~Grimm: \textit{Aperiodic Order: Volume 1, A Mathematical Invitation},
Encyclopedia of Mathematics and its Applications \textbf{149}, Cambridge university press, Cambridge (2014).


\bibitem{BLM}
M.~ Baake, D.~Lenz and  R.V.~Moody: \textit{A characterization of
model sets by dynamical systems}, Ergodic Theory Dynam. Systems
\textbf{27} (2007),  341--382.


\bibitem{BL} M.~Baake and D.~Lenz: \textit{Dynamical systems on translation bounded measures:\ Pure
point dynamical and diffraction spectra}, Ergod. Th. \& Dynam.
Systems  \textbf{24} (2004) 1867--1893.


\bibitem{BM} M.~Baake, R.~Moody (Eds.): \textit{Directions in Mathematical Quasicrystals},
CRM Monogr.\ Ser.\ {\bf 13}, Amer.\ Math.\ Soc., Providence, RI
(2000).


\bibitem{BK} M.~Barge, J.~Kellendonk: \textit{Proximality and pure point spectrum for tiling dynamical systems},
to appear in: Michigan Journal of Mathematics.

\bibitem{BG} L.~Bartholdi and R.~I.~Grigorchuk:
\textit{On the spectrum of Hecke type operators related to some fractal groups},  Tr. Mat. Inst. Steklova
{\bf 231} (2000), Din. Sist., Avtom. i Beskon. Gruppy, 5--45;
translation in {\it Proc. Steklov Inst. Math.} (2000), no. 4 ({\bf
231}), 1--41.

\bibitem{BGN} L.~Bartholdi, R.~I.~Grigorchuk and V.~Nekrashevych: \textit{From fractal groups to fractal sets}
in: 'Fractals in Graz' (P. Grabner and W. Woess editors), Trends in
Mathematics, Birk\"{a}user Verlag, Basel, 2003, 25--118.



\bibitem{BP}
S.~Beckus, F.~Pogorzelski: \textit{Spectrum of Lebesgue measure zero
for Jacobi matrices of quasicrystals},  Math. Phys. Anal. Geom.
\textbf{16} (2013), 289--308.

\bibitem{BBdN} S.~Beckus, J.~Bellissard: \textit{Continuity of the spectrum of a field of self-adjoint
operators}, arXiv:1507.04641.


\bibitem{BBL} A.~Besbes, M.~Boshernitzan, D.~Lenz: \textit{Delone
sets with finite local complexity: Linear repetitivity versus
positivity of weights}, Disc. Comput. Geom. \textbf{49} (2013),
335--347.





\bibitem{Bo} N.~M.~Bon: \textit{Topological full groups of minimal subshifts with subgroups of intermediate growth}, Preprint 2014 (arXiv:1408.0762).





\bibitem{Bosh}
M.~Boshernitzan: \textit{A condition for minimal interval exchange
maps to be uniquely ergodic}, Duke Math.\ J. {\bf  52} (1985),
723--752.


\bibitem{BS} J.~Buescu, I.~Stewart: \textit{Liapunov stability and adding machines},
 Ergodic Theory Dynam. Systems  \textbf{15}  (1995), 271--290.


\bibitem{CFSK}
H.L.~ Cycon, R. G.~Froese, W.~Kirsch, B.~Simon:
\textit{Schr\"odinger operators with application to quantum
mechanics and global geometry}, Texts and Monographs in Physics,
Spinger, Berlin, 1987.


\bibitem{CL} R.~Carmona, J.~Lacroix:  {\em Spectral theory of random {S}chr\"odinger operators},
Birkh\"auser Boston Inc., Boston, MA, 1990.







\bibitem{Dam0} D.~Damanik: \textit{Singular continuous spectrum for a
class of substitution Hamiltonians}, Lett. Math. Phys. \textbf{46}
(1998), 303--311.

\bibitem{Dam} D.~Damanik: \textit{Gordon-type arguments in the spectral theory of one-dimensional
quasicrystals}, in \cite{BM}, 277--305.

\bibitem{DEG} D.~Damanik, M.~Embree, A.~Gorodetski: \textit{Spectral properties of Schr\"odinger operators arising in the study of quasicrystals}, to appear in \cite{KLS}.


\bibitem{DL5} D.~Damanik, D.~Lenz: \textit{The index of Sturmian
sequences},  European J. Combin.  \textbf{23} (2002),  23--29.

\bibitem{DL6} D.~Damanik, D.~Lenz: \textit{Powers in Sturmian sequences},  European J. Combin. \textbf{24}
(2003),  377--390.

\bibitem{DL1} D.~Damanik, D.~Lenz: \textit{A condition of Boshernitzan and uniform convergence in the
multiplicative ergodic theorem},  Duke Math. J. \textbf{133} (2006),
95--123.

\bibitem{DL2} D.~Damanik, D.~Lenz: \textit{Substitution dynamical systems:
characterization of linear repetitivity and applications},  J. Math.
Anal. Appl. \textbf{321} (2006),  766--780.

\bibitem{DZ} D.~Damanik, D.~Zare: \textit{Palindrome complexity bounds for primitive substitution
sequences}, Disc. Math. \textbf{222} (2000), 259--267.

\bibitem{DDMN} D.~D'Angeli, A.~Donno, M.~Matter, T.~Nagnibeda:
\textit{Schreier graphs of the Basilica group},  J. Mod. Dyn.
\textbf{4} (2010),  167--205.






\bibitem{Dek} F.~M. Dekking: \textit{
The spectrum of dynamical systems arising from substitutions of
constant length},  Z. Wahrscheinlichkeitstheorie und Verw. Gebiete
\textbf{41} (1977/78),  221--239.



\bibitem{Dur} F.~Durand: \textit{Linearly recurrent subshifts have a finite number of non-periodic
subshift factors}, Ergod.\ Th.\ \& Dynam.\ Sys. {\bf 20} (2000),
1061--1078.

\bibitem{Dur2} F.~Durand, private communication, 2015.






\bibitem{DHS} F.~Durand, B.~Host, C.~Skau:\textit{Substitution dynamical systems, Bratteli
diagrams and dimension groups}, Ergod.\ Th.\ \& Dynam.\ Sys. {\bf
19} (1999), 953--993.


\bibitem{Fra} D. Francoeur, private communication.

\bibitem{GPS}  T.~Giordano, I.~Putnam, and C.~Skau: \textit{Full
groups of Cantor minimal systems}, Israel J. Math. \textbf{111}
(1999), 285--320.


\bibitem{Gra} L.~Grabowski:\textit{Group ring elements with large spectral density},  Math. Ann.  \textbf{363}  (2015),
637--656.



\bibitem{Gri80} R.~I.~Grigorchuk: \textit{On Burnside's problem on periodic groups. (Russian)}
Funktsional. Anal. i Prilozhen. \textbf{14} (1980), 53 -- 54.

\bibitem{Gri84} R.~I.~Grigorchuk: \textit{Degrees of growth of finitely generated groups and the theory of invariant means (Russian)},
Izv. Akad. Nauk SSSR Ser. Mat. \textbf{48} (1984),  939--985.


\bibitem{Gri11} R.~I.~Grigorchuk:
\textit{Some problems of the dynamics of group actions on rooted trees (Russian)},
Tr. Mat. Inst. Steklova \textbf{273} (2011), Sovremennye Problemy Matematiki, 72--191;
translation in Proc. Steklov Inst. Math. \textbf{273} (2011),  64--175.

\bibitem{GK} R.~I.~Grigorchuk, Ya.~S.~Krylyuk: \textit{The spectral measure of the
Markov operator related to 3-generated 2-group of intermediate
growth and its Jacobi parameters},  Algebra Discrete Math.
\textbf{13} (2012),  237--272.


\bibitem{GLN} R.~I.~Grigorchuk, D.~Lenz, T.~Nagnibeda:
\textit{Spectra of Schreier graphs of Grigorchuk's group and
Schr\"odinger operators with aperiodic order},  arXiv:1412.6822.


\bibitem{GLNS} R.~I.~Grigorchuk, Y.~Leonov, V.~Nekrashevych, V.~Sushchansky:
\textit{Self-similar groups, automatic sequences, and unitriangular
representations}, arXiv:1409.5027.


\bibitem{GN}
R.~I.~Grigorchuk, V.~Nekrashevych: \textit{Self-similar groups,
operator algebras and Schur complements}, Journal of Modern
Dynamics, \textbf{1},  (2007) 323--370.

\bibitem{GNS} R.~I.~Grigorchuk, V.~Nekrashevich, V.~Sushanskii:
\textit{Automata, dynamical systems and infinite groups},  Proc.
Steklov Inst. Math. \textbf{231} (2000), 134-214.

\bibitem{GS1} R.~I.~Grigorchuk, Z.~Sunic: \textit{Asymptotic aspects of Schreier graphs and Hanoi Towers groups},
Comptes Rendus Math. Acad. \textbf{342} (2006),   545--550

\bibitem{GS2} R.~I.~Grigorchuk, Z.~Sunic: \textit{Schreier spectrum of the Hanoi towers group on three pegs},  Proceedings of Symposia in Pure Mathematics, \textbf{77} (2008), 183--198.


\bibitem{GZ} R.~I.~Grigorchuk, A.~Zuk:  \textit{The lamplighter group as a group generated by a 2-state automaton, and its spectrum},
Geom. Dedicata \textbf{87} (2001), 209--244.


\bibitem{GZ3} R.~I.~Grigorchuk,  A.~Zuk.  \textit{The Ihara Zeta function of
infinite graphs, the KNS spectral measure and integrable maps}, in:
Random walks and geometry, V. Kaimanovich (Ed), (2004),  141--180.



\bibitem{Gua1} I.~Guarneri: \textit{Spectral properties of quantum
diffusion on discrete lattices}, Europhys. Lett. \textbf{10} (1989),
95--100.



\bibitem{Gua2} I.~Guarneri:
\textit{On an estimate concerning quantum diffusion in the presence
of a fractal spectrum}, Europhys. Lett. \textbf{21} (1993),
725--733.

\bibitem{GSB} I.~Guarneri, H.~Schulz-Baldes: \textit{Lower bounds on
wave packet propagation by packing dimensions of spectral measures},
Math. Phys. Electron. J. \textbf{5} (1999), Paper 1, 16 pp.



\bibitem{HoJ} R.~Horn, C.R.~Johnson, Matrix Analysis, Cambridge University
Press, Cambridge (1985).

\bibitem{JM} K.~Juschenko and N.~Monod: \textit{Cantor systems, piecewise translations
and simple amenable groups},  Ann. of Math. (2) \textbf{178} (2013),
775--787.



\bibitem{KLS} J.~Kellendonk, D.~Lenz, J.~Savinien (eds): \textit{Directions in aperiodic
order}, to appear in: Progress in Mathematics, Birkh\"auser.


\bibitem{KKT} M.~Kohmoto, L.~P.~ Kadanoff, C.~Tang:\textit{Localization problem in one dimension: mapping
and escape},  Phys. Rev. Lett. \textbf{50} (1983), 1870--1872.

\bibitem{Kot} S.~Kotani: \textit{Jacobi matrices with random potentials taking finitely
many values},  Rev.\ Math.\ Phys. {\bf 1} (1989), 129--133.


\bibitem{LP} J.~Lagarians,  P. A. B. Pleasants: \textit{Repetitive Delone sets and
quasicrystals}, Ergod.\ Th.\ \& Dynam.\ Sys.  {\bf 23} (2003),
831--867.

\bibitem{Las} Y.~Last: \textit{Quantum dynamics and decompositions
of singular continuous spectra}, J. Funct. Anal. \textbf{142}
(1996),  406--445.


\bibitem{LMS}
J.-Y.\ Lee, R.~V.~Moody and B.~Solomyak:  \textit{Pure point
dynamical and diffraction spectra}, Annales Henri Poincar\'{e}
\textbf{3} (2002) 1003--1018.




\bibitem{Lenz} D.~Lenz: \textit{Singular spectrum of Lebesgue measure zero for
one-dimensional quasicrystals},  Comm. Math. Phys. \textbf{227}
(2002), 119--130.


\bibitem{Lenz2} D.~Lenz: \textit{Random operators and crossed
products},  Math. Phys. Anal. Geom. \textbf{2} (1999),  197--220.

\bibitem{Len3} D.~Lenz:\textit{Uniform ergodic theorems on subshifts over a finite alphabet},
 Ergod.\ Th.\ \& Dynam.\ Sys. {\bf 22} (2002), 245--255.

\bibitem{LM}
D.~Lenz and R.~V.~Moody: \textit{Stationary processes with pure
point diffraction}, preprint 2012.



\bibitem{LPV} D.~Lenz, N.~Peyerimhoff, I.~Veselic: \textit{Groupoids, von Neumann Algebras and the Integrated Density of
States}, Math. Phys. Anal. Geom. \textbf{10} (2007), 1--41.


\bibitem{LS} D.~Lenz, N.~Strungaru: \textit{Pure point spectrum for measure dynamical systems on locally
compact Abelian groups},   J. Math. Pures Appl. \textbf{92} (2009)
323--341.

\bibitem{LindM} D.~Lind, B.~Marcus:  \textit{An introduction to symbolic dynamics and coding},
Cambridge University Press, Cambridge, (1995).



\bibitem{Lys} I.~G.~ Lysenok: \textit{A set of defining relations for the Grigorchuk group (Russian)},
Mat. Zametki \textbf{38} (1985),  503--516. English translation in:
Math. Notes \textbf{38} (1985), 784--792.

\bibitem{MBo} N.~Matte~Bon: \textit{Topological full groups of minimal subshifts with subgroups of intermediate growth}, Preprint 2014 (arXiv:1408.0762).

\bibitem{Matui} H. Matui: \textit{Some remarks on topological full groups of Cantor minimal
systems},  Internat. J. Math. \textbf{17} (2006),  231--251.

\bibitem{Moody}
R.~V.~Moody (ed):
 \textit{The Mathematics of Long-Range Aperiodic Order},
NATO-ASI Series C 489, Kluwer, Dordrecht (1997) 239--268.



\bibitem{Nek} V.~Nekrashevych:  {\it Self-similar Groups}, Mathematical Surveys and
Monographs, 117. American Mathematical Society, Providence, RI,
2005.


\bibitem{OPRSS} S.~Ostlund, R.~Pandit, D.~Rand, H.~ Schellnhuber, E.~ Siggia: \textit{One-dimensional Schr\"odinger
equation with an almost periodic potential},  Phys. Rev. Lett. \textbf{50}
 (1983), 1873--1877.


\bibitem{Que} M.~Queffelec: \textit{Substitution dynamical systems -- spectral analysis}, Lecture
Notes in Mathematics, 1294. Springer-Verlag, Berlin, 1987.


\bibitem{Sen} M. Senechal:
\textit{Quasicrystals and geometry}, Cambridge University Press, Cambridge, (1995).


\bibitem{SBGC} D.~Shechtman, I.~Blech, D.~Gratias, J.~W.~Cahn:
\textit{Metallic phase with long-range orientational order
and no translation symmetry},
Phys.\ Rev.\ Lett.\ {\bf 53} (1984) 183--185.



\bibitem{Sol} B.~Solomyak: \textit{Nonperiodicity implies unique composition for self-similar
translationally finite tilings}, Discr.\ Comput.\ Geom. {\bf 20}
(1998), 265--279.



\bibitem{Sun} Z. Sunic: \textit{ Hausdorff dimension in a family of
self-similar groups}, Geometriae Dedicata \textbf{124} (2007),
213--236.

\bibitem{Sut} A.~S\"{u}t\H{o}: \textit{Schr\"odinger difference equation with deterministic ergodic potentials}, in Beyond
Quasicrystals (Les Houches, 1994), Springer, Berlin (1995), 481--549.



\bibitem{Teschl} G.~Teschl, \textit{Jacobi Operators and Completely
Integrable Nonlinear Lattices}, Mathematical Surveys and Monographs
\textbf{72}, Amer. Math. Soc. Providence (2000).




\bibitem{Vor1} Y.~Vorobets: \textit{On a substitution subshift related to
the Grigorchuk group}, preprint 2009, (ArXiv:0910.4800).

\bibitem{Vor2} Y.~Vorobets: \textit{Notes on the Schreier graphs of the Grigorchuk group },
 Dynamical systems and group actions, 221--248, Contemp. Math., 567, Amer. Math. Soc., Providence, RI, 2012.

\bibitem{Wal}
P.~Walters,
\textit{An Introduction to Ergodic Theory},
Springer, New York (1982).

\bibitem{Wei} J.~Weidmann:
\textit{Linear Operators in Hilbert Spaces},  Graduate Texts in
Mathematics  \textbf{68}, Springer-Verlag, New York-Berlin (1980).


\end{thebibliography}
\end{document}